\documentclass{preprint}
\usepackage{preprint}
\usepackage{accents}
\makeatletter
\gdef\@date{}
\let\m@themphpunct\emphpunct
\renewcommand{\emphpunct}[2][3000]{%
  \ifmmode
    \m@themphpunct{#2}%
  \else
    #2\spacefactor#1{}%
  \fi}
\newcommand{\longsimto}{\overset{\textstyle\sim}{\sm@shrelto.3ex\longto}}
\newcommand{\longsubsetto}{\overset{\textstyle\subset\mkern4mu}{\sm@shrelto.5ex\longto}}
\DeclareMathOperator{\clos}{clos}
\DeclareMathOperator{\Ob}{Ob}
\DeclareMathOperator{\Mor}{Mor}
\newcommand{\void}{\bgroup\phantom e\egroup}
\newcommand{\ltangent}[2]{T^{%
  \makebox[.64em][c]{$\scriptstyle\frown$}%
    \kern-.32em
  \makebox[.em][c]{\raisebox{.30ex}{$\scriptscriptstyle\leftrightarrow$}}%
    \kern+.32em}_{#1}#2_{(0)}}
\newcommand{\ttangent}[2]{T^{%
  \makebox[.64em][c]{$\scriptstyle\frown$}%
    \kern-.32em
  \makebox[.em][c]{\raisebox{.23ex}[.5\totalheight][.ex]{$\scriptscriptstyle\updownarrow$}}%
    \kern+.32em}_{#1}#2_{(0)}}
\newcommand{\eff}{\varepsilon}
\newcommand\nefd@t{\bgroup\boldsymbol.\egroup}
\newcommand{\nef}[1]{\accentset{\textstyle\nefd@t}{#1}}
\newcommand{\act}[1]{\accentset{\circlearrowright}{#1}}
\newcommand{\catname}[1]{\mathbf{#1}}
\newcommand{\catquot}[2]{#1_{/#2}}
\newcommand\LGpd@{\catname{LGpd}}
\newcommand\LGpd@dot{\LGpd@^{\boldsymbol\cdot}}
\newcommand{\LGpd}{\@ifstar\LGpd@dot\LGpd@}
\newcommand{\natisoto}{\Rightarrow}
\newcommand{\natcongto}{\overset{\nefd@t\mkern4mu}{\sm@shrelto.8ex\Rightarrow}}
\newcommand{\natiso}{\equiv}
\newcommand{\natcong}{\stackrel\nefd@t\equiv}
\newcommand{\natisoclass}[1]{[#1]}
\newcommand{\natcongclass}[1]{[#1]^{\boldsymbol\cdot}}
\newcommand{\RedLGpd}{\catname{RedLGpd}}
\newcommand{\gpbundleto}{\overset{s=t\mkern4mu}{\sm@shrelto.6ex\longto}}
\newcommand{\gpbundlefro}{\overset{\mkern4mus=t}{\sm@shrelto.6ex\longfro}}
\renewcommand{\Top}{\catname{Top}}
\newcommand{\Skel}{\catname{Skel}}
\newcommand{\effLGpd}{\catname{effLGpd}}
\newcommand{\imbedto}{\longsubsetto}
\newcommand{\efforbGpd}{\catname{efforbGpd}}
\newcommand{\RedOrb}{\catname{RedOrb}}
\makeatother
\theoremstyle{definition}
  \newtheorem{exm}[stmt]{Example}
  \newtheorem*{AxiomI}{\itshape Axiom I}
  \newtheorem*{AxiomII}{\itshape Axiom II}
  \newtheorem*{AxiomIII}{\itshape Axiom III}
\theoremstyle{remark}
  \newtheorem*{rmk*}{Remark}
  \newtheorem*{rmkterm*}{Remark on terminology}
\title{Reduced smooth stacks?%
}%
\author{Giorgio Trentinaglia%
  \thanks{The author acknowledges support %
          from the Portuguese Foundation for Science and Technology %
          ({\it Fun\-da\-\c c\~ao pa\-ra a Ci\-\^en\-cia e a Tec\-no\-lo\-gia}) %
          through grant \#~\mbox{SFRH/BPD/81810/2011}.}
  \\[+1ex]\small\itshape
        Centro de An\'alise Matem\'atica, Geometria e Sistemas Din\^amicos%
  \\[-1ex]\small\itshape
        Instituto Superior T\'ecnico, %
        Av.~Rovisco Pais, 1049-001 Lisbon, Portugal%
}%
\begin{document}
\maketitle

\begin{abstract} \noindent An arbitrary Lie groupoid gives rise to a groupoid of germs of local diffeomorphisms over its base manifold, known as its \emph{effect}. The effect of any bundle of Lie groups is trivial. All quotients of a given Lie groupoid determine the same effect. It is natural to regard the effects of any two Morita equivalent Lie groupoids as being ``equivalent''. In this paper we shall describe a systematic way of comparing the effects of different Lie groupoids. In particular, we shall rigorously define what it means for two arbitrary Lie groupoids to give rise to ``equivalent'' effects. For effective orbifold groupoids, the new notion of equivalence turns out to coincide with the traditional notion of Morita equivalence. Our analysis is relevant to the presentation theory of proper smooth stacks. \end{abstract}

\tableofcontents

\section*{Introduction}

The presentation theorem is a classical result in the theory of orbifolds which essentially dates back to the original papers on ``$V$\mdash manifolds'' by Sa\-ta\-ke and others \cite{Sa,Ka}. (A detailed exposition is given in \cite{MM}.) In modern language \cite{MP,Mo1}, this theorem states that any effective orbifold groupoid%
\footnote{%
See Footnote~\ref{ftn:orbifold} on page~\pageref{ftn:orbifold}.%
} %
is Morita equivalent to the translation groupoid associated to some compact Lie group action on a smooth manifold; of course, the action in question will be effective and have discrete stabilizers. Thus, when regarded as a smooth (De\-ligne\textendash Mum\-ford) stack \cite{Me,Le}, any effective orbifold is isomorphic to a stack of the form \([M/G]\), where \(G\) is a compact Lie group and \(M\) is a smooth manifold on which \(G\) operates smoothly and with discrete stabilizers. The importance of this result nowadays lies principally in the fact that it enables one to reduce the computation of many topological and cohomological invariants of orbifolds to a better understood special case \cite{ALR}. A number of popular research topics in orbifold theory are, in a way or another, related to the presentation theorem. For instance, a long-stand\-ing conjecture affirms that any, say, connected, smooth orbifold stack (effective or not) is of the form \([M/G]\), for \(G\emphpunct,M\) as above. (The reader is referred to \cite{HeM} for an exhaustive discussion and partial results on this conjecture.) For various purposes, for example, for computing Chen\textendash Ruan orbifold cohomology \cite{FaG}, classifying smooth symplectic resolutions of (possibly singular) affine Poisson varieties \cite{GiK}, or building models of conformal field theories on singular spaces \cite{DFMS}, it is important to understand under what conditions an orbifold stack is isomorphic to a ``global quotient'', that is to say, to a stack of the form \([V/G]\), where \(G\) is a finite group of smooth automorphisms of a non-sin\-gu\-lar manifold \(V\). The presentation theorem itself proves to be a useful tool in the study of this kind of problem \cite{ALR}. The presentation theorem also has application in the study of asymptotic spectral properties of elliptic operators on orbifolds \cite{Ko}. The present paper originates from the author's endeavor to extend the presentation theorem beyond the scope of orbifold theory \cite{Tr7} with the intent of gaining a better geometric understanding of general proper smooth stacks and, possibly, laying the foundations for a classification theory of such objects along the lines of \cite{Mo2}.

It turns out that presentation results are by no means special to orbifolds. It is a well-known fact (for a proof of which we refer the reader to Section~5 of \cite{Tr2}) that any Lie groupoid which admits faithful representations is Morita equivalent to the translation groupoid associated to a smooth action of some Lie group on a smooth manifold; when the groupoid is proper, the Lie group can be taken to be compact. The presentation theorem for effective orbifolds is then simply a corollary of this fact and of the fact that any effective orbifold groupoid admits a canonical faithful representation on the tangent bundle of its base manifold. The theorem is thus substantially a result in the representation theory of Lie groupoids. Even though Lie groupoids normally do not admit faithful representations \cite{Tr4,JeM}, in view of recent results obtained by the author \cite{Tr7} it seems plausible that the more ``effectively'' a given proper Lie groupoid acts on its own base the larger is the number of ``interesting'' representations the groupoid possesses. If we agree to call a Lie groupoid \emph{effective} when the only isotropic arrows that have trivial infinitesimal effect (compare Section~\ref{sec:1} below) are the identities, one may conjecture that any effective proper Lie groupoid admits faithful representations. Although at first sight this assertion might look like a good candidate for a generalized presentation theorem, it does not take long to realize that in the non-\'etale case the notion of effective groupoid which we have just introduced is so restrictive that the conjectured result, even if true, would be of little practical utility. [By way of example consider the translation groupoid \(\varGamma=\SO(3)\ltimes\R^3\rightrightarrows\R^3\) associated to the canonical action of the special orthogonal group \(\SO(3)\) on three-di\-men\-sion\-al euclidean space. Since this action is faithful, one would like to say that \(\varGamma\) acts effectively on its base \(\R^3\). On the other hand, outside the origin all the isotropic arrows of \(\varGamma\) have trivial effect.] The approach we propose instead is the following.

Our starting point is the observation that an arbitrary orbifold groupoid can be presented as an extension of some {\em effective} orbifold groupoid by a bundle of finite groups in a canonical fashion: in fact, the ineffective isotropic arrows of an arbitrary orbifold groupoid \(\varGamma\) form a bundle of finite groups \(K\) over the base manifold \(M\) of \(\varGamma\), and the quotient groupoid \(\varGamma/K\rightrightarrows M\) is an effective orbifold groupoid. (Compare e.g.~\cite[\S 2.2]{HeM}.) In general, for an arbitrary short exact sequence of Lie groupoid homomorphisms \(1\to K\emto\varGamma\onto\varGamma'\to1\) presenting a given Lie groupoid \(\varGamma\rightrightarrows M\) as an extension of some other Lie groupoid \(\varGamma'\rightrightarrows M\) by a bundle of Lie groups \(K\), the kernel of the homomorphism \(\varGamma\onto\varGamma'\) consists of ineffective arrows. Heuristically speaking, we may express the circumstance that the two Lie groupoids \(\varGamma\) and \(\varGamma'\) induce the same ``action'' by germs of local diffeomorphisms on their base manifold \(M\) by saying that they give rise to the same ``effective transversal geometry''. Thus, we may liberally rephrase the classical presentation theorem for orbifolds by saying that an arbitrary orbifold groupoid gives rise to the same ``effective transversal geometry'' as the translation groupoid associated to some compact Lie group that acts on a smooth manifold effectively and with discrete stabilizers. Much of the above picture naturally generalizes from orbifold groupoids to arbitrary proper Lie groupoids, in the following manner. Suppose a proper Lie groupoid \(\varGamma\) fits in a short exact sequence of Lie groupoid homomorphisms \(1\to K\emto\varGamma\onto\varGamma'\to1\) where \(\varGamma'\) is a {\em faithfully representable} proper Lie groupoid and \(K\) is a bundle of compact Lie groups. By one of the preceding remarks, if we pick a faithful representation of \(\varGamma'\) and pull it back along the homomorphism \(\varGamma\onto\varGamma'\) to a representation of \(\varGamma\), we obtain what in \cite{Tr4} is called an \emph{effective} representation, i.e.~one whose kernel consists of ineffective arrows. Conversely, it follows from results proven in \cite[\S 4]{Tr4} that an arbitrary proper Lie groupoid which is effectively representable fits in a short exact sequence of the preceding form. On the basis of these considerations, we argue that any positive result about the existence of effective representations ought to be regarded as a presentation theorem of a generalized kind. Allowing ourselves a certain freedom of speech, we might further argue that the primary goal of the presentation theory of proper smooth stacks is to give an explicit characterization of those proper Lie groupoids which give rise to the same ``effective transversal geometries'' as effectively representable proper Lie groupoids or, equivalently, as the translation groupoids associated to compact Lie group actions on smooth manifolds.

The present article is intended to put the heuristic statements involved in the above speculations on a firm mathematical basis. To this end, we are going to propose a conceptual framework for the study of what we shall call ``generalized reduced orbifolds'' or, better, ``reduced smooth stacks''. The basic idea behind our theory is that, instead of trying to define what it means for a smooth stack to be ``reduced'', one should try to rigorously make sense of the assertion that two smooth stacks share the same ``effective transversal geometry''. In the spirit of F.~Klein's {\it Er\-lan\-gen program}, in order to define when two Lie groupoids (regarded as smooth stacks) give rise to the same ``effective transversal geometry'' without having to say what the ``effective transversal geometry'' associated with a Lie groupoid is, we are going to highlight a certain class of Lie groupoid homomorphisms which in a plausible sense preserve all of the groupoid effective transversal structure and then, by a standard categorical localization procedure, declare the members of that class to be isomorphisms. By definition, two Lie groupoids give rise to the same ``effective transversal geometry'' if they turn out to be isomorphic when regarded as objects of the resulting localized category. Of course there are some difficulties which we have to face if we want to make sense of these ideas in a truly satisfactory way: our localized category should admit a \emph{calculus of fractions} \cite{GZ}. This is required, for instance, in order to show that the notion of reduced smooth stack is indeed a generalization of the notion of reduced orbifold (we want the category of reduced orbifolds to imbed into that of reduced smooth stacks as a full subcategory).

We shall now give a sec\-tion-by-sec\-tion description of our article and, with it, some information about the above-men\-tioned localized category. The first two sections are essentially preparatory. In Section~\ref{sec:1} we review some background notions, such as the notion of \emph{effect of an arrow} or the notion of \emph{ineffective isotropy group}, and study how these behave with respect to homomorphisms. In Section~\ref{sec:2} we make a preliminary study of the homomorphisms that will be formally inverted in the process of constructing our localized category, in particular, we explain in what sense these homomorphisms preserve the effective transversal structure of Lie groupoids. In Section~\ref{sec:3}, we describe our prototype category of reduced smooth stacks. At the outset there is the category \(\LGpd*\) with objects all Lie groupoids and with morphisms all Lie groupoid homomorphisms that carry ineffective isotropic arrows into ineffective isotropic arrows. On the morphisms of this category, we introduce an equivalence relation, which we call \emph{natural congruence}, which identifies two homomorphisms when there exists some smooth transformation between them that is natural ``modulo ineffective isotropy''. Next, we form the category \(\catquot{\LGpd*}{\natcong}\) with objects all Lie groupoids and with morphisms all natural congruence classes of homomorphisms in \(\LGpd*\). We show that the homomorphisms studied in Section~\ref{sec:2}, which automatically belong to \(\LGpd*\), project down under the quotient functor \(\LGpd*\to\catquot{\LGpd*}{\natcong}\) to a class of morphisms \(\mathcal E\) which admits a calculus of right fractions. The localized category \(\catquot{\LGpd*}{\natcong}[\mathcal E^{-1}]\) is our prototype category of reduced smooth stacks. We call \emph{effective equivalence} a homomorphism in \(\LGpd*\) whose image under the canonical functor \(\LGpd*\to\catquot{\LGpd*}{\natcong}[\mathcal E^{-1}]\) is an isomorphism. In Section~\ref{sec:4} we show that any effective equivalence induces a homeomorphism at orbit space level and preserves the effective infinitesimal transversal structure at every base point. Moreover, we show that the category of reduced orbifolds imbeds canonically into \(\catquot{\LGpd*}{\natcong}[\mathcal E^{-1}]\). Finally, in Appendix~\ref{sec:A}, we point out a few consequences of an assumption we make on Lie groupoids, namely, second countability, which motivate some basic definitions we give in Sections~\ref{sec:2}\textendash \ref{sec:3}. Among those consequences there are the following two statements\emphpunct: (1)~\em Any fully faithful homomorphism of Lie groupoids that at base level covers the identity map is an isomorphism\em. (2)~\em Any Lie groupoid with just one orbit is transitive\em.

Apart from the formulation of generalized presentation results and the related study of effective representations, which provided the original motivation for our analysis, there is another context where the ideas outlined in the present article are likely to find application, namely, the theory of Riemannian metrics on Lie groupoids propounded recently in \cite{FdH}. (We are indebted to M.~del~Ho\-yo for drawing our attention to this potential utilization of our theory.) The existence of such metrics is a property of Lie groupoids which presumably only depends on the underlying ``reduced smooth stacks'' and, therefore, is invariant under effective equivalence. Our theory may even provide a convenient framework for the analysis of other (e.g.~symplectic or Poisson) ``transversely invariant'' geometric structures on Lie groupoids. The problem of how to give a systematic treatment of such structures is relevant e.g.~to Poisson geometry \cite{FOR}.

We consider the constructions described in this article to be simply a first step towards a full-fledged theory of reduced smooth stacks. In particular, the localized category \(\catquot{\LGpd*}{\natcong}[\mathcal E^{-1}]\) should be regarded simply as a ``minimal working model'' which, albeit already satisfactory from the point of view of our original objectives, may lend itself to further development. Just to mention one possibility, we have contented ourselves with only formally inverting those homomorphisms (among those that preserve the effective transversal structure) for which the standard weak pullback construction suffices to establish the existence of a calculus of fractions. Nothing however excludes that, by suitably modifying that construction, one might be able to invert a larger class of homomorphisms. Ideally, we would like effective equivalences to admit an explicit characterization (at best, one stable under natural congruence), like weak equivalences. We have also deliberately ignored any 2-cat\-e\-gor\-i\-cal aspect in our exposition. However, as argued in \cite{Le}, higher-lev\-el information ought to be taken into account in order to obtain a fully satisfactory theory. The appropriate setting for a theory of ``reduced smooth stacks'' is, most probably, that of bicategories of fractions \cite{Pr}. All these aspects, along with those mentioned in the previous paragraph, shall be addressed elsewhere.

\subsection*{Overall conventions about terminology and notation}

Throughout the article the name \emph{groupoid} will designate a {\em small}\/ category in which all arrows are invertible. A generic groupoid \(\varGamma\) will be written \(\varGamma_{(1)}\rightrightarrows\varGamma_{(0)}\) when there is need to specify its set of objects \(\varGamma_{(0)}\) (also called the \emph{base} of \(\varGamma\)) and its set of arrows \(\varGamma_{(1)}\) (itself often written \(\varGamma\) by abuse of notation) individually. The structure maps of a groupoid \(\varGamma\) will be denoted \(s^\varGamma\) (\emph{source}), \(t^\varGamma\) (\emph{target}), \(m^\varGamma\) (\emph{composition law}), \(u^\varGamma\) (\emph{unit}) and \(i^\varGamma\) (\emph{inverse}), omitting the superscript `\(\varGamma\)' whenever there is no risk of ambiguity. For every pair of objects \(x,y\in\varGamma_{(0)}\) the set of all arrows of source \(x\) respectively target \(y\) will be indicated by \(\varGamma^x=\varGamma(x,-)\) respectively \(\varGamma_y=\varGamma(-,y)\); moreover \(\varGamma^x_y=\varGamma(x,y)\) will denote the intersection \(\varGamma^x\cap\varGamma_y\). The following standard abbreviations will be used systematically: \(sg\) for \(s^\varGamma(g)\)\emphpunct[2000]; \(tg\) for \(t^\varGamma(g)\)\emphpunct[2000]; \(g'g\) for \(m^\varGamma(g',g)\) when \(sg'=tg\)\emphpunct[2000]; \(1_x\) or just \(x\) for \(u^\varGamma(x)\)\emphpunct[2000]; \(g^{-1}\) for \(i^\varGamma(g)\).

By a \emph{differentiable manifold} we mean a (non-emp\-ty) locally compact manifold of class \(C^\infty\). For each point \(x\) of a differentiable manifold \(X\) there is some local chart \(\varphi:U\approxto\R^n\) of class \(C^\infty\) centered at \(x=\varphi^{-1}(0)\in U\). The integer \(n\in\N\), which does not depend on the choice of \(\varphi\), is called the \emph{local dimension} of \(X\) at \(x\) and indicated by \(\dim_xX\). The function \(\dim X:X\to\N\) is locally constant over \(X\). When it is (overall) constant, we say \(X\) is of \emph{constant dimension}. By a \emph{smooth manifold} we mean a differentiable manifold of constant dimension whose topology is Hausdorff and possesses a countable basis of open sets.

A \emph{differentiable groupoid} will be a groupoid \(\varGamma\rightrightarrows X\) in which \(X\) and \(\varGamma\) are differentiable manifolds, \(s^\varGamma\) and \(t^\varGamma\) are {\em submersive} differentiable maps, and the other groupoid structure maps (namely \(m^\varGamma\), \(u^\varGamma\) and \(i^\varGamma\)) are differentiable. A \emph{homomorphism} (of differentiable groupoids) will be a differentiable functor. The term \emph{Lie groupoid} will be regarded as synonymous with \emph{smooth groupoid}; the latter term indicates a differentiable groupoid whose manifold of objects and whose manifold of arrows are both smooth.


\section{Ineffective isotropy}\label{sec:1}

The purpose of this section is to provide a self-con\-tained introduction to some concepts which lie at the heart of the theory expounded in Sections~\ref{sec:2}~to \ref{sec:4}. There is not much claim to originality to be made here. In view of the fundamental role played by the notions to be discussed below, the reader will forgive us if, occasionally, we give full proofs of well-known facts. We will start by reviewing the basic structure theory of differentiable groupoids, in particular, the notion of orbit; in combination with Go\-de\-ment's theorem \cite{Se}, the arguments given in \cite{MM} are essentially still valid in the present, more general context.

Let \(\varGamma\rightrightarrows X\) be an arbitrary differentiable groupoid. For each pair of base points \(x,y\in X\) the subset \(\varGamma^x_y=\varGamma(x,y)\subset\varGamma\) is a differentiable submanifold of \(\varGamma\). In particular, the \emph{isotropy group} \(\varGamma^x_x=\varGamma(x,x)\) has a canonical differentiable group structure. The source fiber \(\varGamma^x=\varGamma(x,-)=s^{-1}(x)\subset\varGamma\) is also a differentiable submanifold of \(\varGamma\). The composition of arrows restricts to a differentiable action of the group \(G_x=\varGamma^x_x\) on the manifold \(\varGamma^x\) from the right. This action is free and has the property that on the quotient set \(\varGamma^x/G_x\) there exists a unique differentiable structure relative to which the canonical projection \(\pr^\varGamma_x:\varGamma^x\onto\varGamma^x/G_x\) becomes a submersion. The differentiable manifold obtained in this way shall be denoted by \(O^\varGamma_x\) and referred to as the \emph{orbit} of \(\varGamma\) (\emph{\(\varGamma\)\mdash orbit}) through \(x\). The action of the group \(G_x\) on the manifold \(\varGamma^x\) makes the differentiable fibration \(\pr^\varGamma_x:\varGamma^x\onto O^\varGamma_x\) into a principal right \(G_x\)\mdash bundle over \(O^\varGamma_x\). Hereafter we will usually write `\([g]\)' in place of `\(\pr^\varGamma_x(g)\)'.

Every orbit \(O^\varGamma_x\) is injectively immersed into the groupoid base \(X\) in a canonical fashion. Namely, there is a unique map \(\incl^\varGamma_x:O^\varGamma_x\to X\) such that the composition \(\incl^\varGamma_x\circ\pr^\varGamma_x\) equals \(t^x:\varGamma^x\to X\), the restriction of the target map to the source fiber, and this map is necessarily differentiable, injective and immersive. The \emph{longitudinal tangent space} at \(x\) is defined to be the image of the tangent linear map \(T_{[x]}\incl^\varGamma_x:T_{[x]}O^\varGamma_x\into T_xX\),
\begin{equation*}
 \ltangent x\varGamma:=\im(T_{[x]}\incl^\varGamma_x)\subset T_xX
\end{equation*}
(the base point \(x\) and the corresponding unit arrow \(1_x\) being identified notationally). The \emph{transversal tangent space} at \(x\) is defined to be the quotient vector space
\begin{equation*}
 \ttangent x\varGamma:=\quot*T_xX\by\ltangent x\varGamma\endquot\textpunct.
\end{equation*}
Intuitively, the former space consists of all those vectors in \(T_xX\) that are ``tangent'' to the orbit \(O^\varGamma_x\), whereas the latter consists of those that are ``perpendicular'' to it.

We observe that for every arrow \(g\in\varGamma\)
\begin{equation}
\label{eqn:12B.16.2}
 \im(T_{[tg]}\incl_{tg})=\im(T_{[g]}\incl_{sg})\textpunct.
\end{equation}
To see this, notice that the right-trans\-la\-tion map \(h\mapsto hg\) is a diffeomorphism of \(\varGamma^{tg}\) onto \(\varGamma^{sg}\) which is equivariant relative to the group isomorphism \(\varGamma^{tg}_{tg}\simto\varGamma^{sg}_{sg}\) given by conjugation by \(g^{-1}\) and therefore descends to a well-de\-fined diffeomorphism between the orbits \(O_{tg}\) and \(O_{sg}\) which intertwines the maps \(\incl_{tg}\) and \(\incl_{sg}\).

Let an arrow \(g\in\varGamma\) be given. Put \(x=sg\) and \(x'=tg\). Consider an arbitrary local differentiable section \(\gamma:U\emto\varGamma\) to the source map \(s:\varGamma\onto X\) through \(g=\gamma(x)\). The tangent linear map \(T_x(t\circ\gamma):T_xX\to T_{x'}X\) carries the longitudinal subspace \(\ltangent x\varGamma\subset T_xX\) into its counterpart \(\ltangent{x'}\varGamma\subset T_{x'}X\); this follows from the existence of a differentiable map \(c:\incl_x^{-1}(U)\to O_{x'}\) such that \(\incl_{x'}\circ c=(t\circ\gamma)\circ\incl_x\) [such a map exists as a corollary to the local triviality of the orbit fibration \(\pr_x:\varGamma^x\onto O_x\)]. As a consequence, \(T_x(t\circ\gamma)\) must induce a well-de\-fined linear map between the transversal tangent spaces \(\ttangent x\varGamma\) and \(\ttangent{x'}\varGamma\); we let this map be indicated by \(\eff^\gamma_x\) provisionally.

\begin{lem}\label{lem:12B.16.1} The linear map\/ \(\eff^\gamma_x:\ttangent x\varGamma\to\ttangent{x'}\varGamma\) does not depend on the choice of a local source section\/ \(\gamma\) through the given arrow\/ \(g\in\varGamma(x,x')\). \end{lem}

\begin{proof} Let \(\gamma_1\) and \(\gamma_2\) be local \(s\)\mdash sections through \(g=\gamma_1(x)=\gamma_2(x)\). We have
\begin{equation*}
 T_x(t\circ\gamma_1)-T_x(t\circ\gamma_2)=T_gt\circ(T_x\gamma_1-T_x\gamma_2)\textpunct.
\end{equation*}
The difference \(T_x\gamma_1-T_x\gamma_2\) takes values in the \(s\)\mdash vertical tangent space \(\ker T_gs\). Now, since \(T_gt^x=T_{[g]}\incl_x\circ T_g\pr_x\), it is clear from the identity~\eqref{eqn:12B.16.2} that \(T_gt\) maps \(\ker T_gs\) into \(\ltangent{x'}\varGamma\). Thus the two linear maps \(T_x(t\circ\gamma_i)\sidetext[i=1,2]\) differ from one another by a linear map taking values in \(\ltangent{x'}\varGamma\) and hence induce the same map between the transversal tangent spaces \(\ttangent x\varGamma\) and \(\ttangent{x'}\varGamma\). \end{proof}

We set \(\eff(g):=\eff^\gamma_x\) and call this the (\emph{infinitesimal}) \emph{effect} of \(g\). Notice that \(\eff(1_x)=\id\) for every base point \(x\in X\) and that \(\eff(g'g)=\eff(g')\circ\eff(g)\) for every composable pair of arrows \((g',g)\in\varGamma\ftimes st\varGamma\). Both identities are an immediate consequence of the definitions.

\begin{lem}\label{lem:12B.16.2} For any base point\/ \(x\in X\), the correspondence that to each arrow\/ \(g\in\varGamma^x_x\) associates its effect\/ \(\eff(g)\) gives rise to a differentiable group homomorphism of\/ \(\varGamma^x_x\) into\/ \(\GL(\ttangent x\varGamma)\) which shall be indicated by\/ \(\eff_x\). \end{lem}

\begin{proof} There is only to check the differentiability of the correspondence \(\eff_x\). Put \(G=\varGamma^x_x\) for brevity. Consider an arbitrary local \(C^\infty\) source trivialization \(\varphi:\varOmega\approxto U\times\R^n\emphpunct,g\mapsto\varphi(g)=\bigl(sg,f(g)\bigr)\) [where \(\varOmega\subset\varGamma\) and \(U\subset X\) are open subsets]. Put \(A=G\cap\varOmega\) and let \(\psi:A\times U\to X\) be the map given by \(\psi(a,u)=t\varphi^{-1}\bigl(u,f(a)\bigr)\). Clearly \(\psi\) is of class \(C^\infty\) and the same is true of its second partial \(\D_2\psi(-,x):A\times T_xX\to T_xX\). From the definitions it follows that if we fix any vec\-tor-space basis \(\varv_1,\dotsc,\varv_m\) of \(T_xX\) so that the last \(m-r\) vectors span the longitudinal subspace \(\ltangent x\varGamma\) then the top left \(r\times r\)~minor in the matrix representing \(\D_2\psi(a,x)\in\End(T_xX)\) will be the matrix representing the effect of \(a\) relative to the induced basis \(\bar\varv_1,\dotsc,\bar\varv_r\) of the transversal tangent space \(\ttangent x\varGamma\). \end{proof}

We shall refer to the closed subgroup
\begin{equation*}
 \nef\varGamma^x_x:=\ker\bigl(\eff_x:\varGamma^x_x\to\GL(\ttangent x\varGamma)\bigr)\subset\varGamma^x_x
\end{equation*}
as the \emph{ineffective} isotropy group of \(\varGamma\) at \(x\).

We proceed to study the behavior of ineffective isotropy under homomorphisms. Let \(\phi:\varGamma\to\varDelta\) be an arbitrary homomorphism of differentiable groupoids. Let \(X\) denote the base manifold of \(\varGamma\) and \(Y\) that of \(\varDelta\). Also let \(f:X\to Y\) denote the base map induced by \(\phi\). For each base point \(x\in X\) we have a map \(O^\phi_x:O^\varGamma_x\to O^\varDelta_{f(x)}\) characterized by the equation%
\begin{subequations}
\label{eqn:12B.3.6}
\begin{equation}
\label{eqn:12B.3.6a}
 O^\phi_x\circ\pr^\varGamma_x=\pr^\varDelta_{f(x)}\circ\phi^x\textpunct,
\end{equation}
where \(\phi^x:\varGamma^x\to\varDelta^{f(x)}\) indicates the map induced by \(\phi\) between the source fibers. This map is necessarily differentiable (because \(\pr^\varGamma_x\) is a surjective submersion and therefore admits local sections). Alternatively, \(O^\phi_x\) can be characterized through the equation
\begin{equation}
\label{eqn:12B.3.6b}
 \incl^\varDelta_{f(x)}\circ O^\phi_x=f\circ\incl^\varGamma_x\textpunct.
\end{equation}
\end{subequations}
If we differentiate the latter equation at \([x]\in O^\varGamma_x\), we obtain the identity
\begin{equation*}
 T_{[f(x)]}\incl^\varDelta_{f(x)}\circ T_{[x]}O^\phi_x=T_xf\circ T_{[x]}\incl^\varGamma_x\textpunct,
\end{equation*}
which makes it evident that the tangent linear map \(T_xf:T_xX\to T_{f(x)}Y\) carries the longitudinal subspace \(\ltangent x\varGamma\subset T_xX\) into the longitudinal subspace \(\ltangent{f(x)}\varDelta\subset T_{f(x)}Y\) and hence yields a well-de\-fined linear map of the transversal tangent space \(\ttangent x\varGamma\) into the transversal tangent space \(\ttangent{f(x)}\varDelta\); this map shall be denoted by \(\ttangent x\phi\) hereafter.

\begin{lem}\label{lem:12B.16.3} The linear map\/ \(\ttangent x\phi:\ttangent x\varGamma\to\ttangent{f(x)}\varDelta\) intertwines the group representations\/ \(\eff^\varGamma_x:\varGamma^x_x\to\GL(\ttangent x\varGamma)\) and\/ \(\eff^\varDelta_{f(x)}:\varDelta^{f(x)}_{f(x)}\to\GL(\ttangent{f(x)}\varDelta)\) via the homomorphism\/ \(\phi^x_x:\varGamma^x_x\to\varDelta^{f(x)}_{f(x)}\). More explicitly, the following identity holds for every arrow\/ \(g\in\varGamma^x_x\):
\begin{equation}
\label{eqn:12B.16.6}
 \ttangent x\phi\circ\eff^\varGamma_x(g)=\eff^\varDelta_{f(x)}\bigl(\phi(g)\bigr)\circ\ttangent x\phi\textpunct.
\end{equation} \end{lem}

\begin{proof} Let us set \(y=f(x)\) and \(h=\phi(g)\) for brevity. Let \(\gamma\) be any local source section of class \(C^\infty\) through \(g\) and let \(\delta\) be any similar section through \(h\). We have
\begin{align*}
 T_xf\circ T_x(t^\varGamma\circ\gamma)-T_y(t^\varDelta\circ\delta)\circ T_xf
&=T_x(f\circ t^\varGamma\circ\gamma)-T_x(t^\varDelta\circ\delta\circ f)
\\%
&=T_x(t^\varDelta\circ\phi\circ\gamma)-T_x(t^\varDelta\circ\delta\circ f)
\\%
&=T_ht^\varDelta\circ\bigl(T_x(\phi\circ\gamma)-T_x(\delta\circ f)\bigr)\textpunct.
\end{align*}
The difference \(T_x(\phi\circ\gamma)-T_x(\delta\circ f):T_xX\to T_h\varDelta\) is a linear map taking values in the \(s^\varDelta\)\mdash vertical subspace \(\ker T_hs^\varDelta\subset T_h\varDelta\), because
\begin{align*}
 T_hs^\varDelta\circ\bigl(T_x(\phi\circ\gamma)-T_x(\delta\circ f)\bigr)
&=T_x(s^\varDelta\circ\phi\circ\gamma)-T_x(s^\varDelta\circ\delta\circ f)
\\%
&=T_x(f\circ s^\varGamma\circ\gamma)-T_x(\id\circ f)
\\%
&=T_xf-T_xf=0\textpunct.
\end{align*}
The identity~\eqref{eqn:12B.16.2} implies that \(T_ht^\varDelta\) carries \(\ker T_hs^\varDelta\) into \(\ltangent y\varDelta\), whence \eqref{eqn:12B.16.6}. \end{proof}

\begin{prop}\label{cor:12B.16.4} The following implications hold for any homomorphism of differentiable groupoids\/ \(\phi:\varGamma\to\varDelta\) for every base point\/ \(x\) of\/ \(\varGamma\).
\begin{enumerate} \def\labelenumi{\upshape(\alph{enumi})} \itemsep=0pt
 \item If\/ \(\ttangent x\phi\) is surjective then\/ \(\phi(\nef\varGamma^x_x)\subset\nef\varDelta^{\phi x}_{\phi x}\).
 \item If\/ \(\ttangent x\phi\) is injective then\/ \(\phi^{-1}(\nef\varDelta^{\phi x}_{\phi x})\cap\varGamma^x_x\subset\nef\varGamma^x_x\).
 \item If\/ \(\ttangent x\phi\) is bijective then, for all\/ \(g\in\varGamma^x_x\), \(g\in\nef\varGamma^x_x\aeq\phi(g)\in\nef\varDelta^{\phi x}_{\phi x}\).
\end{enumerate} \end{prop}

\begin{proof} The last assertion follows from the other two, which in turn are straightforward consequences of the identity~\eqref{eqn:12B.16.6}. \end{proof}

We shall call a homomorphism of differentiable groupoids \(\phi:\varGamma\to\varDelta\) \emph{transversal} whenever the following map is a submersion (notations as above).
\begin{equation}
\label{eqn:12B.16.7}
 s^\varDelta\circ\pr_2:X\ftimes ft\varDelta\to Y
\end{equation}

\begin{prop}\label{prop:12B.16.5} Let\/ \(\phi:\varGamma\to\varDelta\) be a transversal homomorphism of differentiable groupoids. Then for each base point\/ \(x\) of\/ \(\varGamma\) the linear map\/ \(\ttangent x\phi:\ttangent x\varGamma\to\ttangent{\phi x}\varDelta\) is surjective. \end{prop}

\begin{proof} As before, let \(f:X\to Y\) denote the map induced by \(\phi\) on the groupoid bases. Also set \(y=f(x)\). By the transversality hypothesis on \(\phi\), over some open neighborhood \(V\) of \(y\) in \(Y\) it will be possible to find a \(C^\infty\) section through \((x,1_y)\in X\ftimes ft\varDelta\) to the map~\eqref{eqn:12B.16.7}. This section will be of the form \((a,\delta)\) with \(a:V\to X\) a \(C^\infty\) map and \(\delta:V\to\varDelta\) a differentiable \(s^\varDelta\)\mdash section through \(1_y\) such that \(t^\varDelta\circ\delta=f\circ a\). The difference
\begin{equation*}
 T_xf\circ T_ya-\id_{T_yY}=T_y(t^\varDelta\circ\delta)-\id_{T_yY}
\end{equation*}
will be a linear map carrying \(T_yY\) into \(\ltangent y\varDelta\), because \(T_y(t^\varDelta\circ\delta)\) represents \(\eff^\varDelta_y(1_y)=\id\). The surjectivity of \(\ttangent x\phi\) is evident now. \end{proof}

Recall that a homomorphism of differentiable groupoids \(\phi:\varGamma\to\varDelta\) is said to be a \emph{weak equivalence} if the associated map~\eqref{eqn:12B.16.7} is a surjective submersion (so that in particular \(\phi\) is transversal) and the square diagram
\begin{equation}
\label{eqn:12B.16.8}
 \begin{split}
\xymatrix{%
 \varGamma
 \ar[d]^{(s,t)}
 \ar[r]^\phi
 &
	\varDelta
	\ar[d]^{(s,t)}
\\
 X\times X
 \ar[r]^{f\times f}
 &
	Y\times Y
}\end{split}
\end{equation}
(in the above notations) is a pullback within the category of differentiable manifolds.

\begin{prop}\label{prop:12B.16.6} Let\/ \(\phi:\varGamma\to\varDelta\) be a weak equivalence of differentiable groupoids. Then for each base point\/ \(x\) of\/ \(\varGamma\) the linear map\/ \(\ttangent x\phi:\ttangent x\varGamma\to\ttangent{\phi x}\varDelta\) is bijective. \end{prop}

\begin{proof} Let \(f\), \(y\) and \((a,\delta):V\to X\ftimes ft\varDelta\) be as in the proof of the previous proposition. There is a (necessarily unique) \(C^\infty\) map from the open subset \(\incl_y^{-1}(V)\) of \(O_y\) into \(O_x\) whose composition with \(\incl_x\) coincides with \(\incl_y\sidetext[\text{restricted to }\incl_y^{-1}(V)]\) followed by \(a\); this can be seen by exploiting the availability of local \(C^\infty\) sections to the projection \(\pr_y:\varDelta^y\onto O_y\) together with the pullback universal property of the diagram~\eqref{eqn:12B.16.8}. The existence of such a map implies that \(T_ya:T_yY\to T_xX\) carries \(\ltangent y\varDelta\) into \(\ltangent x\varGamma\) and therefore induces a linear map of \(\ttangent y\varDelta\) into \(\ttangent x\varGamma\) which we may call say \(\alpha\). From the proof of Proposition~\ref{prop:12B.16.5} it follows that \(\alpha\) must be a right inverse for \(\ttangent x\phi\).

We contend that \(\alpha\) is also a left inverse for \(\ttangent x\phi\). To see this, choose any open neighborhood \(U\) of \(x\) in \(X\) so that \(f(U)\subset V\) and let \(f_U:U\to V\) denote the map induced by \(f\) by restriction. The pullback universal property of the diagram~\eqref{eqn:12B.16.8} applied to the maps \(\delta\circ f_U\) and \((\id,a\circ f_U)\) yields a local source section \(\gamma:U\to\varGamma\) of class \(C^\infty\) through \(1_x\) such that \(t\circ\gamma=a\circ f_U\). The composite linear map \(T_ya\circ T_xf=T_x(t\circ\gamma)\), thus, represents the trivial effect \(\eff(1_x)=\id\). \end{proof}

Combining the above with Proposition~\ref{cor:12B.16.4}(c), we obtain the following notable property of weak equivalences, which appears already in \cite{Tr4} as part of Lemma~4.2.

\begin{cor}\label{cor:advances} For any weak equivalence of differentiable groupoids\/ \(\phi:\varGamma\to\varDelta\) and for any base point\/ \(x\) of\/ \(\varGamma\), the group isomorphism\/ \(\phi^x_x:\varGamma^x_x\simto\varDelta^{\phi x}_{\phi x}\) establishes a bijection between the ineffective subgroup\/ \(\nef\varGamma^x_x\) of\/ \(\varGamma^x_x\) and the ineffective subgroup\/ \(\nef\varDelta^{\phi x}_{\phi x}\) of\/ \(\varDelta^{\phi x}_{\phi x}\). \end{cor}

We conclude the section with a remark about natural transformations which will be needed only later in Section~\ref{sec:4} on a single occasion. Recall that a natural transformation \(\tau\) between two homomorphisms of differentiable groupoids \(\phi,\psi:\varGamma\to\varDelta\), in mathematical notation `\(\tau:\phi\natisoto\psi\)', is a map of class \(C^\infty\) from the base \(X\) of \(\varGamma\) into the arrows of \(\varDelta\) which to each point \(x\in X\) assigns an arrow \(\tau(x)\in\varDelta(\phi x,\psi x)\) in such a way as to yield an ordinary natural transformation of (abstract) functors between \(\phi\) and \(\psi\).

\begin{lem}\label{lem:natural} Let\/ \(\tau:\phi\natisoto\psi\) be a natural transformation between two homomorphisms of differentiable groupoids\/ \(\phi,\psi:\varGamma\to\varDelta\). For every base point\/ \(x\) of\/ \(\varGamma\)
\begin{equation}
\label{eqn:natural}
 \eff\bigl(\tau(x)\bigr)\circ\ttangent x\phi=\ttangent x\psi\textpunct.
\end{equation} \end{lem}

\begin{proof} Pick any local \(s^\varDelta\)\mdash section \(\delta\) through \(\tau(x)=\delta(\phi x)\). Then
\begin{align*}
 T_{\phi x}(t\circ\delta)\circ T_x\phi-T_x\psi
&=T_{\phi x}(t\circ\delta)\circ T_x\phi-T_x(t\circ\tau)
\\%
&=Tt\circ(T_{\phi x}\delta\circ T_x\phi-T_x\tau)
\end{align*}
will be a linear map taking values in the longitudinal tangent space \(\ltangent{\psi x}\varDelta\), because the parenthesized linear map takes values in \(\ker Ts\). \end{proof}

\section{Completely transversal and full homomorphisms}\label{sec:2}

We open the present section\textemdash which like the previous one is preparatory\textemdash by reviewing some standard basic constructions. Then, we proceed to establish some seemingly not so well-known results, notably Proposition~\ref{prop:14A.1.3} and Lemma~\ref{lem:14A.1.8}, which motivate and underlie the theory discussed in the subsequent sections.

Let \(\varGamma\rightrightarrows X\), \(\varDelta\rightrightarrows Y\) be differentiable groupoids. We remind the reader that we call a homomorphism \(\phi:\varGamma\to\varDelta\) \emph{transversal} if the associated map \(s\circ\pr_2:X\ftimes\phi t\varDelta\to Y\) is submersive. We call \(\phi\) \emph{completely} transversal if the same map is, moreover, surjective.%
\footnote{\label{ftn:transversal}%
In the literature, a homomorphism which is completely transversal is usually called \emph{essentially surjective}. We believe our terminology is preferable, for obvious reasons.}

\begin{lem}\label{lem:14A.1.1} The composition of two (completely) transversal homomorphisms of differentiable groupoids is also (completely) transversal. \end{lem}

\begin{proof} Let \(\phi:\varGamma\to\varDelta\) and \(\psi:\varDelta\to\varSigma\) be arbitrary homomorphisms of differentiable groupoids. Let \(X\), \(Y\), \(Z\) denote the base manifolds of \(\varGamma\), \(\varDelta\), \(\varSigma\) respectively. We have the following commutative diagram of \(C^\infty\) maps
\begin{equation*}
 \begin{split}
\xymatrix{%
 (X\ftimes\phi t\varDelta)\ftimes{\psi\circ s\circ\pr_2}t\varSigma
 \ar[d]^{(s\circ\pr_2)\times\id}
 \ar@{->>}[r]
 &
	X\ftimes{\psi\circ\phi}t\varSigma
	\ar[d]^{s\circ\pr_2}
\\
 Y\ftimes\psi t\varSigma
 \ar[r]^-{s\circ\pr_2}
 &
	Z
}\end{split}
 \qquad
 \begin{split}
\xymatrix@C=.6em{%
 (x,h;k)
 \ar@{|->}[d]
 \ar@{|->}[r]
 &
	(x,\psi(h)k)
	\ar@{|->}[d]
\\
 (sh,k)
 \ar@{|->}[r]
 &
	sk=s(\psi(h)k)
}\end{split}
\end{equation*}
in which the top horizontal map is a surjective submersion. The claim about the composition \(\psi\circ\phi\) is obvious now. \end{proof}

Let \(\varGamma\rightrightarrows X\), \(\varDelta\rightrightarrows Y\) and \(\varSigma\rightrightarrows B\) be differentiable groupoids. Let \(\phi:\varGamma\to\varSigma\) and \(\psi:\varDelta\to\varSigma\) be differentiable groupoid homomorphisms. Assume that \(\phi\) is transversal. Then we can form the \emph{weak pullback} of \(\phi\) along \(\psi\)
\begin{equation}
\label{eqn:14A.1.2}
 \begin{split}
\xymatrix@C=3em{%
 \varDelta\mathbin{_\psi\sqcap_\phi}\varGamma
 \ar[d]^{\pr_\varDelta}
 \ar[r]^-{\pr_\varGamma}
 &
	\varGamma
	\ar[d]^\phi_(.3){}="s"
\\
 \varDelta
 \ar[r]^-\psi^(.7){}="t"
 \ar@{=>}@/_/"s";"t"
 &
	\varSigma
}\end{split}
\end{equation}
whose construction we proceed to recall. Because of the transversality of \(\phi\), the following turn out to be differentiable manifolds:%
\begin{align*}
 Z
&=Y\ftimes\psi t\varSigma\ftimes s\phi X=Y\ftimes\psi{t\circ\pr_1}(\varSigma\ftimes s\phi X)\textpunct;
\\%
 \varPi
&=\varDelta\ftimes{\psi\circ s}t\varSigma\ftimes s{\phi\circ s}\varGamma=\varDelta\ftimes{\psi\circ s}{t\circ\pr_1\circ(\id\times s)}(\varSigma\ftimes s{\phi\circ s}\varGamma)\textpunct.
\end{align*}
Regarding \(Z\) as the base manifold and \(\varPi\) as the manifold of arrows, one declares the map \(s\times\id\times s\) given by \((h,k,g)\mapsto(sh,k,sg)\) [which is clearly a surjective submersion] to be the source, and the map \((h,k,g)\mapsto\bigl(th,\psi(h)k\phi(g)^{-1},tg\bigr)\) [which is obviously \(C^\infty\)] to be the target. The composition law, unit and inversion map are respectively given by
\begin{equation*}
 \begin{split}
 (h',k',g')(h,k,g)=(h'h,k,g'g)\textpunct,\quad 1(y,k,x)&=(1_y,k,1_x)
\\%
 \text{and}\quad (h,k,g)^{-1}&=\bigl(h^{-1},\psi(h)k\phi(g)^{-1},g^{-1}\bigr)\textpunct;
 \end{split}
\end{equation*}
they are obviously all \(C^\infty\) and make \(\varPi\rightrightarrows Z\) into a differentiable groupoid which we agree to indicate by \(\varDelta\mathbin{_\psi\sqcap_\phi}\varGamma\).

\begin{lem}\label{lem:14A.1.2} In the weakly commutative diagram~\eqref{eqn:14A.1.2}, the homomorphism\/ \(\pr_\varDelta\) given by\/ \((h,k,g)\mapsto h\) is transversal and in fact submersive at the level of groupoid bases. Moreover, if\/ \(\phi\) is completely transversal then\/ \(\pr_\varDelta\) is onto at the level of bases and thus a~fortiori also completely transversal. \qed \end{lem}

Let \(\varDelta\rightrightarrows Y\) be a differentiable groupoid. Let \(f:X\to Y\) be a differentiable map which is \emph{essentially submersive} in the sense that the associated map \(s\circ\pr_2:X\ftimes ft\varDelta\to Y\) is submersive. Then we can form the \emph{pullback groupoid}
\begin{equation}
\label{eqn:12B.21.1}
 f^*\varDelta:=(X\ftimes ft\varDelta)\ftimes{s\circ\pr_2}fX\rightrightarrows X
\end{equation}
whose groupoid structure is uniquely determined by the requirement that the obvious `projection' \(\pi:f^*\varDelta\to\varDelta\) ought to be a homomorphism of groupoids. The `projection' \(\pi\) is in fact a homomorphism of differentiable groupoids, and a weak equivalence as soon as \(f\) is also \emph{essentially surjective} (i.e.~as soon as \(s\circ\pr_2\) is also surjective).

\begin{prop}\label{prop:14A.1.3} Let\/ \(\varGamma\), \(\varDelta\) be {\em Lie} groupoids. Let\/ \(\phi:\varGamma\to\varDelta\) be a homomorphism which is transversal and full (as an abstract functor). Then\/ \(\phi\) is automatically\/ \emph{\(C^\infty\)\mdash full} in the following sense; let\/ \(X\), \(Y\) denote the base manifolds of\/ \(\varGamma\), \(\varDelta\) respectively and let\/ \(f:X\to Y\) denote the base map induced by\/ \(\phi\). For any maps\/ \(h:U\to\varDelta\) and\/ \((x,x'):U\to X\times X\) of class\/ \(C^\infty\) which satisfy the condition\/ \((s,t)\circ h=(f\times f)\circ(x,x')\) there exist an open cover\/ \(c:U'\onto U\) (i.e.~a surjective map which on each connected component of\/ \(U'\) restricts to a diffeomorphism onto an open subset of\/ \(U\)) and a\/ \(C^\infty\) map\/ \(g:U'\to\varGamma\) such that\/ \(\phi\circ g=h\circ c\) and such that\/ \((s,t)\circ g=(x,x')\circ c\).
\begin{equation}
\label{eqn:14A.1.5}
 \begin{split}
\xymatrix@R=1.7ex{%
 U'
 \ar@{-->}@(dr,l)[drr]^(.3)g
 \ar@{-->>}[r]^c
 &
	U
	\ar@(d,l)[dddr]_{(x,x')}
	\ar@(r,ul)[drr]^h
\\
 &&
		\varGamma
		\ar[dd]^{(s,t)}
		\ar[r]^\phi
		&
			\varDelta
			\ar[dd]^{(s,t)}
\\ \\
 &&
		X\times X
		\ar[r]^{f\times f}
		&
			Y\times Y
}\end{split}
\end{equation} \end{prop}

\begin{proof} Since the map \(f:X\to Y\) is essentially submersive, we can form the pullback groupoid \(f^*\varDelta\rightrightarrows X\), as in \eqref{eqn:12B.21.1}. In the case under consideration this is a {\em Lie} groupoid. Its `projection' \(\pi:f^*\varDelta\to\varDelta\) is a homomorphism of Lie groupoids which covers the map \(f:X\to Y\) at base level. By the evident \(C^\infty\)\mdash pullback universal property of this homomorphism, \(\phi\) will factor through \(f^*\varDelta\) as \(\phi=\pi\circ\phi'\) for a unique homomorphism \(\phi':\varGamma\to f^*\varDelta\) covering the identity on the common base \(X\) of \(\varGamma\) and \(f^*\varDelta\).
\begin{equation}
\label{eqn:14A.1.6}
 \begin{split}
\xymatrix@R=1.7ex{%
 \varGamma
 \ar@{-->}[dr]^(.6){\phi'}
 \ar@(r,ul)[drr]^(.6)\phi
 \ar@(d,l)[dddr]_{(s,t)}
\\
 &
	f^*\varDelta
	\ar[dd]^{(s,t)}
	\ar[r]^-\pi
	&
		\varDelta
		\ar[dd]^{(s,t)}
\\ \\
 &
	X\times X
	\ar[r]^{f\times f}
	&
		Y\times Y
}\end{split}
\end{equation}
Fullness of \(\phi\) plainly entails fullness of \(\phi'\), so that \(\phi'\) is a surjective homomorphism of Lie groupoids which covers the identity over \(X\). It follows from Proposition~\ref{prop:12B.5.2} in the appendix that \(\phi'\) is actually an {\em epimorphism} of Lie groupoids viz.~that the map which \(\phi'\) induces between the manifolds of arrows is a surjective submersion.

Now, suppose we are assigned maps \(h\) and \((x,x')\) as in the universal problem~\eqref{eqn:14A.1.5}. By the \(C^\infty\)\mdash pullback universal property of the square~\eqref{eqn:14A.1.6}, there will be a unique \(C^\infty\) map \(h':U\to f^*\varDelta\) satisfying the conditions \(\pi\circ h'=h\) and \((s,t)\circ h'=(x,x')\). We are thus reduced to the simpler problem depicted below, where the outer square commutes by definition of \(h'\).
\begin{equation}
\label{eqn:14A.1.7}
 \begin{split}
\xymatrix@R=1.7ex{%
 U'
 \ar@{-->}@(dr,l)[drr]^(.3)g
 \ar@{-->>}[r]
 &
	U
	\ar@(d,l)[dddr]_{(x,x')}
	\ar@(r,ul)[drr]^{h'}
\\
 &&
		\varGamma
		\ar[dd]^{(s,t)}
		\ar@{->>}[r]^-{\phi'}
		&
			f^*\varDelta
			\ar[dd]^{(s,t)}
\\ \\
 &&
		X\times X
		\ar[r]^=
		&
			X\times X
}\end{split}
\end{equation}
Consider any point \(u\in U\). Since \(\phi':\varGamma\onto f^*\varDelta\) is a surjective submersion, we can find a local \(C^\infty\) section \(g'_u:V_u\to\varGamma\) to \(\phi'\) defined around \(h'(u)\). We can then select an open neighborhood \(W_u\) of \(u\) in \(U\) so that \(h'(W_u)\subset V_u\) and define a map \(g_u\) on \(W_u\) by setting \(g_u=g'_u\circ\restr h'\to W_u\endrestr\). The coproduct map \(g=\coprod g_u\) [defined on the disjoint union \(U'=\coprod W_u\) of all the open sets \(W_u\) as \(u\) is let vary over \(U\)] will be a solution to the problem schematized in \eqref{eqn:14A.1.7}. \end{proof}

\begin{cor}\label{cor:14A.1.4} Let\/ \(\phi:\varGamma\to\varDelta\) be a completely transversal homomorphism of Lie groupoids which is both full and faithful (as an abstract functor). Then\/ \(\phi\) must be a weak equivalence. \qed \end{cor}

Of course, the property of \(C^\infty\)\mdash fullness defined with Proposition~\ref{prop:14A.1.3} makes sense for homomorphisms between arbitrary differentiable groupoids. Any \(C^\infty\)\mdash full homomorphism of differentiable groupoids is full.

\begin{lem}\label{lem:14A.1.5} The composition of any two\/ \(C^\infty\)\mdash full homomorphisms of differentiable groupoids is itself\/ \(C^\infty\)\mdash full. \qed \end{lem}

\begin{lem}\label{lem:14A.1.6} In the weak pullback diagram~\eqref{eqn:14A.1.2}, the homomorphism\/ \(\pr_\varDelta\) must be\/ \(C^\infty\)\mdash full whenever so is\/ \(\phi\) (of course assuming this is transversal). \end{lem}

\begin{proof} Suppose two \(C^\infty\) maps \(h:U\to\varDelta\) and \((y,k,x;y',k',x'):U\to Z\times Z\) [recall: \(Z=Y\ftimes\psi t\varSigma\ftimes s\phi X\) in the notations of \ref{lem:14A.1.2}] are given which satisfy the condition
\begin{equation*}
 (s,t)\circ h=(\pr_\varDelta\times\pr_\varDelta)\circ(y,k,x;y',k',x')=(y,y')\textpunct.
\end{equation*}
Then the map \([i\circ k'][\psi\circ h]k:U\to\varSigma\mskip.mu plus 5mu\sidetext(i=\text{ inversion map in $\varSigma$ hereafter})\) given by \(u\mapsto(k'_u)^{-1}\psi(h_u)k_u\) and the map \((x,x'):U\to X\times X\) will satisfy the condition
\begin{equation*}
 (s,t)\circ([i\circ k'][\psi\circ h]k)=(\phi\times\phi)\circ(x,x')\textpunct.
\end{equation*}
Now, since \(\phi\) is \(C^\infty\)\mdash full, it will be possible to find an open cover \(c:U'\onto U\) for which there exists a \(C^\infty\) map \(g:U'\to\varGamma\) such that
\begin{equation*}
 \phi\circ g=([i\circ k'][\psi\circ h]k)\circ c
\quad
 \text{and}
\quad
 (s,t)\circ g=(x,x')\circ c\textpunct.
\end{equation*}
The map \((h\circ c,k\circ c,g):U'\to\varDelta\ftimes{\psi\circ s}t\varSigma\ftimes s{\phi\circ s}\varGamma\) will be a solution for the original universal problem expressing the \(C^\infty\)\mdash fullness of the homomorphism \(\pr_\varDelta\). \end{proof}

\begin{rmk*} Making \(U=\{\ast\}\) in the preceding proof, the same reasoning shows that \(\pr_\varDelta\) must be full whenever so is \(\phi\). \end{rmk*}

The above lemmas suggest that one might be able to build a reasonable category of fractions by localizing differentiable groupoids at their completely transversal, \(C^\infty\)\mdash full homomorphisms. This is indeed so, as we will see shortly. Before proceeding further, however, we must convince ourselves that such homomorphisms preserve the ``effective transversal geometry'' of differentiable groupoids. This is the goal of the next couple of lemmas. Recall that any differentiable groupoid \(\varGamma\rightrightarrows X\) gives rise to an associated {\em orbit space} \(X/\varGamma\); this is the quotient of \(X\) by the equivalence relation that identifies any two points which can be connected by an arrow in \(\varGamma\), topologized with the finest topology that makes the quotient projection \(X\to X/\varGamma\) continuous.

\begin{lem}\label{lem:14A.1.7} Any full, completely transversal homomorphism of differentiable groupoids induces a homeomorphism between the associated orbit spaces. \end{lem}

\begin{proof} Let \(\phi:\varGamma\to\varDelta\) be any such homomorphism and let \(f:X/\varGamma\to Y/\varDelta\) denote the continuous map induced by \(\phi\) between the orbit spaces of \(\varGamma\) and \(\varDelta\). By complete transversality of \(\phi\), \(f\) must be surjective and open. By fullness of \(\phi\), \(f\) must be injective. \end{proof}

For any base point \(x\) of an arbitrary differentiable groupoid \(\varGamma\), put
\begin{equation*}
 \act\varGamma^x_x:=\quot*\varGamma^x_x\by\nef\varGamma^x_x\endquot\textpunct;
\end{equation*}
since \(\nef\varGamma^x_x\) is a closed normal subgroup of the differentiable group \(\varGamma^x_x\), the quotient group \(\act\varGamma^x_x\) will inherit a canonical differentiable group structure. We shall refer to the effective \(\act\varGamma^x_x\)~space \(\ttangent x\varGamma\) as the \emph{effective first order approximation} of \(\varGamma\) at \(x\) or as the \emph{effective infinitesimal model} for \(\varGamma\) at \(x\).

A homomorphism of differentiable groupoids \(\phi:\varGamma\to\varDelta\) shall be called \emph{faithfully} transversal if (it is transversal and) for every point \(x\) in the base of \(\varGamma\) the linear map \(\ttangent x\phi:\ttangent x\varGamma\to\ttangent{\phi x}\varDelta\) is injective (hence bijective by Proposition~\ref{prop:12B.16.5}). By Corollary~\ref{cor:12B.16.4}, any transversal \(\phi:\varGamma\to\varDelta\) must induce a differentiable group homomorphism \(\act\phi^x_x:\act\varGamma^x_x\to\act\varDelta^{\phi x}_{\phi x}\) at every \(x\), which will be injective as soon as \(\phi\) is faithfully transversal; by Proposition~\ref{prop:12B.5.5} in the appendix, any injective homomorphism of differentiable groups must be a monomorphism.

\begin{lem}\label{lem:14A.1.8} Any\/ \(C^\infty\)\mdash full, transversal homomorphism of differentiable groupoids is faithfully transversal. \end{lem}

\begin{proof} Let \(\phi:\varGamma\to\varDelta\) be any such homomorphism. Let us temporarily assume that \(\varGamma\) and \(\varDelta\) are groupoids over the same base manifold \(X\) and that \(\phi\) covers the identity over \(X\). Then for any given point \(x\in X\) the linear map \(\ttangent x\phi\) will be injective if and only if the tangent base map \(T_x\phi=\id_{T_xX}\) at \(x\) satisfies the condition
\begin{equation*}
 (T_x\phi)\delta x=\delta x\in\ltangent x\varDelta
\quad \seq \quad
 \delta x\in\ltangent x\varGamma
\end{equation*}
for all \(\delta x\in T_xX\). Thus, suppose \(\delta x\in\ltangent x\varDelta\), i.e., \(\delta x=\dd\tau th_\tau\) for some \(C^\infty\) path \(\tau\mapsto h_\tau\in\varDelta(x,-)\) such that \(h_0=1_x\). Since \(\phi\) is \(C^\infty\)\mdash full, we can lift \(\tau\mapsto h_\tau\) locally around zero to a \(C^\infty\) path \(\tau\mapsto g_\tau\in\varGamma(x,-)\); it will not be restrictive to assume that \(g_0=1_x\) (for otherwise we can simply take \(\tau\mapsto g_\tau g_0^{-1}\)). Then
\begin{equation*}
\textstyle%
 \delta x=\dd\tau th_\tau=\dd\tau t\phi(g_\tau)=\dd\tau tg_\tau\in\ltangent x\varGamma\textpunct.
\end{equation*}

Now suppose \(\phi:\varGamma\to\varDelta\) is completely general. As in the proof of Proposition~\ref{prop:14A.1.3}, we can write \(\phi\) as the composition of a homomorphism \(\phi'\) covering the identity on the bases with a weak equivalence \(\pi\). Using the faithfulness of \(\pi\), it is straightforward to check that \(\phi'\) must be itself \(C^\infty\)\mdash full. Since weak equivalences are always faithfully transversal by Proposition~\ref{prop:12B.16.6}, we are finally reduced to the special situation considered at the beginning. \end{proof}

Clearly, any \(C^\infty\)\mdash full homomorphism \(\phi:\varGamma\to\varDelta\) will induce an epimorphism of differentiable groups \(\phi^x_x:\varGamma^x_x\onto\varDelta^{\phi x}_{\phi x}\) at each base point \(x\) of \(\varGamma\). Thus, by the remarks preceding the last lemma, when \(\phi\) is also transversal the quotient homomorphism \(\act\phi^x_x:\act\varGamma^x_x\to\act\varDelta^{\phi x}_{\phi x}\) will be an isomorphism (of differentiable groups); therefore, by the lemma, we will have an isomorphism
\begin{equation}
\label{eqn:14A.1.9}
 (\act\phi^x_x,\ttangent x\phi):(\act\varGamma^x_x,\ttangent x\varGamma)\longsimto(\act\varDelta^{\phi x}_{\phi x},\ttangent{\phi x}\varDelta)
\end{equation}
between the effective infinitesimal model for \(\varGamma\) at \(x\) and that for \(\varDelta\) at \(\phi x\).

\section{The category of reduced Lie groupoids}\label{sec:3}

Throughout the rest of the article we shall let \(\LGpd\) stand for the category of Lie groupoids and Lie groupoid homomorphisms. We shall use the notation `\(h_1\equiv h_2\pmod{\nef\varDelta}\)' as an abbreviation for `\em\(sh_1=sh_2=y\), \(th_1=th_2\) and\/ \(h_2^{-1}h_1\in\nef\varDelta^y_y\)\em'. Since \(\nef\varDelta\) is a normal,%
\footnote{\label{ftn:normal}%
That is to say, for each \(h\in\varDelta\), letting \(c_h:\varDelta^{sh}_{sh}\simto\varDelta^{th}_{th}\) indicate the group isomorphism \(k\mapsto hkh^{-1}\) (conjugation by \(h\)), we have \(c_h(\nef\varDelta^{sh}_{sh})\subset\nef\varDelta^{th}_{th}\).%
} %
totally isotropic (abstract) subgroupoid of \(\varDelta\), it follows that the binary relation \(\mathord\equiv\pmod{\nef\varDelta}\) thus defined on the arrows of \(\varDelta\) is a (categorical, abstract) congruence (although in general not a {\em regular} congruence in the sense of Appendix~\ref{sec:A}).

\begin{defn}\label{defn:14A.1.9} Let \(\varGamma\mathrel{%
\xymatrix@1@C=1.3em@M=.em{%
 \ar@<+.5ex>[r]^\phi
 \ar@<-.5ex>[r]_\psi
 &
}}\varDelta\) be any pair of homomorphisms between two given Lie groupoids \(\varGamma\) and \(\varDelta\). By a \emph{natural congruence} \(\tau\) between \(\phi\) and \(\psi\), in symbols `\(\tau:\phi\natcongto\psi\)', we shall mean a map \(\tau:X\to\varDelta\) of class \(C^\infty\) from the base manifold \(X\) of \(\varGamma\) into the manifold of arrows of \(\varDelta\) such that \(\tau(x)\in\varDelta(\phi x,\psi x)\) for all \(x\in X\) and such that for all \(g\in\varGamma\)
\begin{equation*}
 \tau(tg)\phi(g)\equiv\psi(g)\tau(sg)\pmod{\nef\varDelta}\textpunct.
\end{equation*} \end{defn}

Obviously, any ordinary natural isomorphism \(\tau:\phi\natisoto\psi\) is a~fortiori a natural congruence \(\tau:\phi\natcongto\psi\). The collection of all natural congruences between homomorphisms \(\varGamma\to\varDelta\) is closed under the obvious operations of composition \(\tau'\tau\) and inversion \(\tau^{-1}\). On each hom-set \(\LGpd(\varGamma,\varDelta)\) the binary relation
\begin{equation*}
 \phi\natcong\psi
\quad
 \aeq
\quad
 \text{\em there exists a natural congruence between\/ $\phi$ and\/ $\psi$\em}
\end{equation*}
is therefore an equivalence. If we now let \(\LGpd*\) denote the subcategory of \(\LGpd\) with the same objects (Lie groupoids) and with morphisms all those \(\phi\in\LGpd(\varGamma,\varDelta)\) such that \(\phi^x_x(\nef\varGamma^x_x)\subset\nef\varDelta^{\phi x}_{\phi x}\) for every base point \(x\) of \(\varGamma\), it is immediate to check that the above equivalence \(\natcong\) gives rise to a categorical congruence on \(\LGpd*\). We shall let \(\catquot{\LGpd*}{\natcong}\) denote the resulting quotient category whose morphisms are the \(\natcong\)\mdash equivalence classes of morphisms in \(\LGpd*\) (compare \cite[II.8]{MacLane}). The notation \(\natcongclass\phi\) will be used to indicate the \(\natcong\)\mdash class of a homomorphism \(\phi\in\Mor(\LGpd*)\).

Let \(\mathcal E\subset\Mor(\LGpd*)\) denote the collection of all the (\(C^\infty\)\mdash)full and completely transversal homomorphisms of Lie groupoids (we know that any transversal homomorphism must lie within \(\LGpd*\), by the results of Section~\ref{sec:1}). By abuse of notation, we shall use the same letter to indicate the image of \(\mathcal E\) under the quotient projection functor \(\LGpd*\to\catquot{\LGpd*}{\natcong}\emphpunct,\phi\mapsto\natcongclass\phi\); the intended meaning will always be clear from the context. We claim that the localized category \(\catquot{\LGpd*}{\natcong}[\mathcal E^{-1}]\) admits a calculus of right fractions. The remainder of the present section will essentially be devoted to checking that the pertinent axioms \cite[I.2.2]{GZ} are indeed satisfied for \(\mathcal E\). Before doing this, however, we want to illustrate (and justify) the preceding definitions by means of a few examples.

\begin{exm}\label{exm:1} The goal of our first example is to show that there are normally many homomorphisms of Lie groupoids which do not lie in the subcategory \(\LGpd*\), and thus that passing from \(\LGpd\) to \(\LGpd*\) already makes a relevant difference from the point of view of what ``reduced smooth stacks'' turn out to be eventually in comparison with ordinary smooth stacks.

Let the orthogonal group \(\GO(2)\) act on \(\R^2\) by matrix multiplication (canonical action). Similarly, let the special orthogonal group \(\SO(3)\) act on \(\R^3\) by matrix multiplication. Let \(G\subset\SO(3)\) denote the closed subgroup formed by all matrices \(P\in\SO(3)\) such that \(Pe_3=\pm e_3\), where \(e_3\) indicates the standard basis vector \((0,0,1)\) in \(\R^3\). We have a continuous and injective group homomorphism \(\theta:\GO(2)\emto G\) given by \[A\mapsto\within[%
\begin{array}{c|c}
 A & 0
\\
\hline
 0 & \det A
\end{array}]\textpunct,\] the determinant \(\det:\GO(2)\to\{\pm1\}\) being itself a continuous homomorphism. Obviously, the map \(f:\R^2\emto\R^3\) given by \((x,y)\mapsto(x,y,0)\) is equivariant with respect to \(\theta\). The Lie groupoid homomorphism%
\footnote{\label{ftn:Ex1}%
Recall that if \(G\) is a differentiable group acting say from the left on a differentiable manifold \(X\) in a \(C^\infty\) fashion then one can form the corresponding \emph{translation groupoid} \(G\ltimes X=G\times X\rightrightarrows X\). This is the differentiable groupoid whose source map is the projection from \(G\times X\) on \(X\), target map is the group action, and arrow composition law is given by the formula \((g',gx)(g,x)=(g'g,x)\). Clearly, when \(G\) is a Lie group and \(X\) is a smooth manifold, \(G\ltimes X\) is a Lie groupoid.}
\[\theta\ltimes f:\GO(2)\ltimes\R^2\longto G\ltimes\R^3\] sends the ineffective isotropic arrow \((\within(%
\begin{smallmatrix}
 1 & 0
\\
 0 & -1
\end{smallmatrix});1,0)\in\GO(2)\ltimes\R^2\) to the isotropic arrow \((\within(%
\begin{smallmatrix}
 1 & 0  & 0
\\
 0 & -1 & 0
\\
 0 & 0  & -1
\end{smallmatrix});1,0,0)\in G\ltimes\R^3\), whose effect is evidently non-triv\-i\-al since the \(G\)\mdash orbit through \((1,0,0)\) is the circle \(\setofall*(x,y,0)\suchthat x^2+y^2=1\endsetofall\). \end{exm}

\begin{exm}\label{exm:2} The next examples demonstrate that the relation of natural congruence is quite coarser than the relation of natural isomorphism in most cases. We shall exhibit several pairs of homomorphisms \(\phi,\psi\in\Mor(\LGpd*)\) such that \(\phi\) is naturally congruent but not isomorphic to \(\psi\).

{\bf(a)}\emphpunct{} Let \(\SO(2)\) act canonically on \(\R^2\), and let \(\SO(3)\) act similarly on \(\R^3\). Consider the map \(f:\R^2\to\R^3\) defined by \((x,y)\mapsto(0,0,1)\). This map is obviously equivariant relative to the Lie group homomorphism \(\theta:\SO(2)\emto\SO(3)\) given by \(A\mapsto\within[%
\begin{array}{c|c}
 A & 0
\\
\hline
 0 & 1
\end{array}]\). For any Lie group endomorphism \(\eta\) of \(\SO(2)\), the same map is also \((\theta\circ\eta)\)\mdash equivariant. The two Lie groupoid homomorphisms [which, trivially, lie in \(\Mor(\LGpd*)\)] \[\theta\ltimes f\emphpunct,(\theta\circ\eta)\ltimes f:\SO(2)\ltimes\R^2\longto\SO(3)\ltimes\R^3\] are naturally congruent, because all the isotropic arrows in \(\SO(3)\ltimes\R^3\) outside the origin are ineffective, but not naturally isomorphic (unless of course \(\eta=\id\)), because all the isotropy groups in \(\SO(3)\ltimes\R^3\) outside the origin are abelian.

{\bf(b)}\emphpunct{} Let \(\omega:\R\to\R\) be an arbitrary \(C^\infty\) real-val\-ued function of one real variable. Let \(\R=(\R,+)\sidetext[=\text{ the additive group of the real numbers}]\) act on the product \(\C\times\R\) by operating on the first factor by rotations with frequency \(\omega\)\emphpunct: \({\theta\cdot(z,t)}=(e^{i\omega(t)\theta}z,t)\). Any \(C^\infty\) real-val\-ued function of one real variable \(\varphi:\R\to\R\) such that \(\abs{\varphi(t)}=\abs{\omega(t)}\) for all \(t\) gives rise to a Lie groupoid homomorphism \[\R\ltimes[\C\times\R]\longto\SO(3)\ltimes\R^3\emphpunct,(\theta;z,t)\mapsto(\within(%
\begin{smallmatrix}
 \cos{\varphi(t)\theta} & \within"-\sin{\varphi(t)\theta}" & 0
\\[.4ex]
 \sin{\varphi(t)\theta} &          \cos{\varphi(t)\theta}  & 0
\\[.4ex]
 0                      & 0                                & 1
\end{smallmatrix});0,0,t)\] which belongs to \(\Mor(\LGpd*)\) and whose kernel contains the totally isotropic, normal subgroupoid \(K\) of \(\R\ltimes[\C\times\R]\) defined by the expression
\begin{equation*}
 K=\setofall*(\theta;z,t)\suchthat[\omega(t)\neq0\et\omega(t)\theta\in2\pi\Z]\vel\theta=0\endsetofall\textpunct.
\end{equation*}
Since \(K\) is, in fact, a closed regular kernel of constant dimension in \(\R\ltimes[\C\times\R]\) (compare Appendix~\ref{sec:A} and Example~\ref{exm:4} below), this homomorphism factors through the quotient groupoid \(\varGamma=(\R\ltimes[\C\times\R])/K\rightrightarrows\C\times\R\), thus giving rise to a Lie groupoid homomorphism \(\phi:\varGamma\to\SO(3)\ltimes\R^3\) which still belongs to \(\Mor(\LGpd*)\). Now, let \(\varphi_0\), \(\varphi_1\) be any two functions as above, and let \(\phi_0\), \(\phi_1\) denote the Lie groupoid homomorphisms \(\varGamma\to\SO(3)\ltimes\R^3\) they give rise to. Clearly, one has \(\phi_0\natcong\phi_1\) as soon as \(\varphi_0(0)=\varphi_1(0)\). However \(\phi_0\not\natiso\phi_1\) unless \(\varphi_0(t)=\varphi_1(t)\) for all \(t\neq0\). \end{exm}

\begin{exm}\label{exm:3} Let \(\varGamma\gpbundleto M\gpbundlefro\varDelta\) be any two Lie group bundles over the same base manifold \(M\). Any two homomorphisms \(\phi,\psi:\varGamma\to\varDelta\) covering the identity map on the base are naturally congruent. In particular, any endomorphism \(\phi:\varGamma\to\varGamma\) covering the identity map on the base is naturally congruent to the identity homomorphism \(\id:\varGamma\to\varGamma\). It follows that any two Lie group bundles over the same base manifold are isomorphic when regarded as objects of the category \(\catquot{\LGpd*}{\natcong}\). \end{exm}

\begin{npar}\label{npar:axioms} We proceed to check that the class of morphisms \(\mathcal E\subset\Mor(\catquot{\LGpd*}{\natcong})\) satisfies the axioms for a calculus of right fractions \cite[I.2.2]{GZ}:

\begin{AxiomI} \em The class of morphisms\/ \(\mathcal E\subset\Mor(\catquot{\LGpd*}{\natcong})\) is multiplicative viz.~contains the identities and is closed under composition\em. This is obvious a~fortiori, since \(\mathcal E\) is multiplicative already as a subclass of \(\Mor(\LGpd*)\). Cf.~Lemmas~\ref{lem:14A.1.1}~and \ref{lem:14A.1.5}. \end{AxiomI}

\begin{AxiomII} \em Given any pair of morphisms\textup{*}
\begin{equation*}
 \begin{split}
\xymatrix@C=3em@R=4.2ex{%
 &
	\varGamma'
	\ar@{~>}[d]^{\natcongclass\phi\in\mathcal E}
\\
 \varDelta
 \ar[r]^{\natcongclass\psi}
 &
	\varGamma
}\end{split}
\quad
 \text{ there exists a commutative diagram}
\quad
 \begin{split}
\xymatrix@C=3em@R=4.2ex{%
 \varDelta'
 \ar@{~>}[d]|(.559){}^{\natcongclass{\phi'}\in\mathcal E}
 \ar@{-->}[r]^{\smash[t]{\natcongclass{\psi'}}}
 &
	\varGamma'
	\ar@{~>}[d]
\\
 \varDelta
 \ar[r]
 &
	\varGamma
}\end{split}
\end{equation*}\em
(*for visual immediacy, we shall be using wavy arrows to represent morphisms belonging to \(\mathcal E\)). Since \(\phi\) belongs to \(\mathcal E\), by definition it must be transversal so that we may form the weak pullback (which in our case is clearly also a {\em smooth} groupoid)
\begin{equation*}
\xymatrix@C=3em{%
 \varDelta\mathbin{_\psi\sqcap_\phi}\varGamma'
 \ar[d]^{\pr_\varDelta}
 \ar[r]^-{\pr_{\varGamma'}}
 &
	\varGamma'
	\ar[d]^\phi_(.3){}="s"
\\
 \varDelta
 \ar[r]^-\psi^(.7){}="t"
 \ar@{=>}@/_/"s";"t"
 &
	\varGamma
}\end{equation*}
(the diagram here is supposed to commute up to natural isomorphism). By the lemmas \ref{lem:14A.1.2}~and \ref{lem:14A.1.6}, \(\phi\in\mathcal E\) implies \(\pr_\varDelta\in\mathcal E\). Thus, we will be done if we set \(\varDelta'=\varDelta\mathbin{_\psi\sqcap_\phi}\varGamma'\), \(\phi'=\pr_\varDelta\) and \(\psi'=\pr_{\varGamma'}\), provided we show that \(\pr_{\varGamma'}\) belongs to \(\Mor(\LGpd*)\). To this end, suppose \(h'\) is an ineffective isotropy arrow in \(\varDelta\mathbin{_\psi\sqcap_\phi}\varGamma'\). Since both \(\psi\) and \(\pr_\varDelta\) belong to \(\Mor(\LGpd*)\), so will do their composition \(\psi\circ\pr_\varDelta\) and hence also the composite homomorphism \(\phi\circ\pr_{\varGamma'}\) as this is naturally isomorphic to \(\psi\circ\pr_\varDelta\).%
\footnote{\label{ftn:AxII}%
More in general, whenever \(\phi\natcong\psi\in\LGpd(\varGamma,\varDelta)\) and \(\phi\in\LGpd*(\varGamma,\varDelta)\) then also \(\psi\in\LGpd*(\varGamma,\varDelta)\). The proof is immediate.%
} %
We will then have \(\phi(\pr_{\varGamma'}h')\in\nef\varGamma\). Now \(\phi\) is faithfully transversal by Lemma~\ref{lem:14A.1.8}. We can therefore invoke Corollary~\ref{cor:12B.16.4} to yield the desired conclusion that \(\pr_{\varGamma'}h'\in\nef\varGamma'\). \end{AxiomII}

\begin{AxiomIII} \em Given any morphisms
\begin{equation*}
\xymatrix@C=3.1em{%
 \varGamma
 \ar@<+.5ex>[r]^{\natcongclass{\psi_1}}
 \ar@<-.5ex>[r]_{\natcongclass{\psi_2}}
 &
	\varDelta
	\ar@{~>}[r]^-{\natcongclass\phi\in\mathcal E}
	&
		\varDelta'
}\end{equation*}
with the property that\/ \(\natcongclass\phi\natcongclass{\psi_1}=\natcongclass\phi\natcongclass{\psi_2}\), there is also some morphism\/ \(\natcongclass\pi:\varGamma'\mathrel{%
\xymatrix@1@C=1.3em@M=.em{%
 \ar@{~>}[r]
 &
}}\varGamma\) belonging to\/ \(\mathcal E\) with the property that\/ \(\natcongclass{\psi_1}\natcongclass\pi=\natcongclass{\psi_2}\natcongclass\pi\)\em. Let \(X\), \(Y\) and \(Y'\) respectively denote the base manifolds of the Lie groupoids \(\varGamma\), \(\varDelta\) and \(\varDelta'\). By assumption, we have \(\phi\circ\psi_1\natcong\phi\circ\psi_2\), in other words, there exists some natural congruence \(\tau':\phi\circ\psi_1\natcongto\phi\circ\psi_2\). We make use of the \(C^\infty\)\mdash fullness of \(\phi\) in the situation depicted below.
\begin{equation}
\label{eqn:14A.1.11}
 \begin{split}
\xymatrix@R=1.7ex{%
 X'
 \ar@{-->}@(dr,l)[drr]^(.3)\tau
 \ar@{-->>}[r]^c
 &
	X
	\ar@(d,l)[dddr]_{(\psi_1,\psi_2)}
	\ar@(r,ul)[drr]^{\tau'}
\\
 &&
		\varDelta
		\ar[dd]^{(s,t)}
		\ar[r]^-\phi
		&
			\varDelta'
			\ar[dd]^{(s,t)}
\\ \\
 &&
		Y\times Y
		\ar[r]^-{\phi\times\phi}
		&
			Y'\times Y'
}\end{split}
\end{equation}
The map \(c:X'\onto X\) is an open cover, in particular, a surjective submersion. Hence it certainly makes sense to pull back the Lie groupoid \(\varGamma\) along it.
\begin{equation}
\label{eqn:14A.1.12}
 \begin{split}
\xymatrix{%
 c^*\varGamma
 \ar[d]^{(s,t)}
 \ar[r]^-\pi
 &
	\varGamma
	\ar[d]^{(s,t)}
\\
 X'\times X'
 \ar[r]^-{c\times c}
 &
	X\times X
}\end{split}
\end{equation}
The resulting pullback groupoid, denoted \(\varGamma'=c^*\varGamma\rightrightarrows X'\), will be a {\em Lie} groupoid as well,%
\footnote{\label{ftn:AxIII}%
It can always be assumed that \(X'\) is a smooth manifold. Indeed, in any case \(X'\) is a (non-emp\-ty) Hausdorff manifold of constant dimension, just because so is \(X\). Since \(X\) is second countable, we can always find a countable open cover \(X''\onto X\) subordinate to \(X'\onto X\) via some map \(X''\to X'\).%
} %
and its canonical projection \(\pi\) onto \(\varGamma\) will be a weak equivalence of Lie groupoids and thus, a~fortiori, an element of \(\mathcal E\) and hence a morphism in the category \(\LGpd*\). We contend that the \(C^\infty\) map \(\tau\) in \eqref{eqn:14A.1.11} must be a natural congruence \(\psi_1\circ\pi\natcongto\psi_2\circ\pi\). Certainly \((s,t)\circ\tau=(\psi_1,\psi_2)\circ c=(\psi_1\circ\pi,\psi_2\circ\pi)\) by the commutativity of \eqref{eqn:14A.1.11} and \eqref{eqn:14A.1.12}. Moreover, for every arrow \(g'\in\varGamma'\),
\begin{align}
\notag%
 \phi\bigl(\tau(tg')\psi_1(\pi g')\bigr)
&=\tau'(t\pi g')[\phi\circ\psi_1](\pi g')
\\
\notag%
&\equiv[\phi\circ\psi_2](\pi g')\tau'(s\pi g')\pmod{\nef\varDelta'}
\\
\label{eqn:14A.1.13}
&=\phi\bigl(\psi_2(\pi g')\tau(sg')\bigr)\textpunct.
\end{align}
By Lemma~\ref{lem:14A.1.8}, \(\phi\) must be a faithfully transversal homomorphism. We immediately conclude from Corollary~\ref{cor:12B.16.4} that the following implication holds:
\begin{equation*}
 \phi(h_1)\equiv\phi(h_2)\pmod{\nef\varDelta'}
\quad
 \seq
\quad
 h_1\equiv h_2\pmod{\nef\varDelta}\textpunct.
\end{equation*}
We thus deduce from \eqref{eqn:14A.1.13} that
\begin{equation*}
 \tau(tg')[\psi_1\circ\pi](g')\equiv[\psi_2\circ\pi](g')\tau(sg')\pmod{\nef\varDelta}\textpunct,
\end{equation*}
as desired. \end{AxiomIII} \end{npar}

\begin{npar}\label{npar:14A.1.10} Our proof that the multiplicative system \(\mathcal E\) in the category \(\catquot{\LGpd*}{\natcong}\) does indeed admit a calculus of right fractions is thus finished. For convenience, we recall how the corresponding explicit model for the localization \(\catquot{\LGpd*}{\natcong}[\mathcal E^{-1}]\) is obtained.

Let us set \(\mathfrak L=\catquot{\LGpd*}{\natcong}\) for brevity. One starts by introducing the ``category of right fractions'' \(\mathfrak L\mathcal E^{-1}\):
\begin{itemize} \itemsep=0pt
 \item The objects of \(\mathfrak L\mathcal E^{-1}\) are simply those of \(\mathfrak L\) (namely Lie groupoids);
 \item The morphisms in \(\mathfrak L\mathcal E^{-1}(\varGamma,\varDelta)\) are equivalence classes of ``spans'' in \(\mathfrak L\)
\begin{equation*}
 \xymatrix@C=2.5em{%
 \varDelta
 &
	\varGamma'
	\ar[l]_-\alpha
	\ar@{~>}[r]^-\varepsilon
	&
		\varGamma
}\qquad \text{with $\varepsilon\in\mathcal E$,}
\end{equation*}
 two such spans \((\alpha_1,\varepsilon_1)\) and \((\alpha_2,\varepsilon_2)\) being equivalent whenever there is a commutative diagram
\begin{equation*}
 \begin{split}
\xymatrix@C=4.3em@R=3.4ex{%
 &
	\varGamma'_1
	\ar[dl]_{\alpha_1}
	\ar@{~>}[dr]^{\varepsilon_1}
\\
 \varDelta
 &
	\varGamma''
	\ar[u]
	\ar[d]
	\ar@{~>}[r]^(.4){\varepsilon'}
	&
		\varGamma
\\
 &
	\varGamma'_2
	\ar[ul]^{\alpha_2}
	\ar@{~>}[ur]_{\varepsilon_2}
}\end{split}\qquad \text{with $\varepsilon'\in\mathcal E$;}
\end{equation*}
 \item The composition of morphisms in \(\mathfrak L\mathcal E^{-1}\) is given by \(%
 \def\objectstyle{\scriptstyle}%
 \vcenter
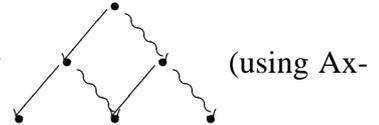
{%
 \xymatrix@1@C=1em@R=2.5ex{%
  &     &
		\bullet
		\ar[dl] \ar@{~>}[dr]
 \\
  &
	\bullet
	\ar[dl] \ar@{~>}[dr]
	&       &
			\bullet
			\ar[dl] \ar@{~>}[dr]
 \\
  \bullet
  &     &
		\bullet
		&       &
				\bullet
 }}\) (using Axiom~II).
\end{itemize}
Let us indicate by \(\alpha/\varepsilon\) the class of a span \((\alpha,\varepsilon)\) in \(\mathfrak L\mathcal E^{-1}(\varGamma,\varDelta)\). There is a canonical functor from the category of fractions \(\mathfrak L\mathcal E^{-1}\) into the localization \(\mathfrak L[\mathcal E^{-1}]\); it is the identity on objects, and it sends \(\alpha/\varepsilon\) to \(\pr(\alpha)\pr(\varepsilon)^{-1}\), where \(\pr:\mathfrak L\to\mathfrak L[\mathcal E^{-1}]\) denotes the universal localization functor. This functor \(\mathfrak L\mathcal E^{-1}\to\mathfrak L[\mathcal E^{-1}]\) is full and faithful \cite[Proposition~2.4, p.~14]{GZ} and hence an isomorphism of categories. \end{npar}

\begin{defn}\label{defn:14A.1.11} We shall call \(\catquot{\LGpd*}{\natcong}[\mathcal E^{-1}]\) the category of \emph{reduced Lie groupoids} and use the shorthand \(\RedLGpd\) for it. \end{defn}

The reader might be wondering whether it is possible to give a more ``geometric'' characterization of the category which we have just defined (say, up to equivalence). For instance, one might consider the idea of focusing attention on some collection of smooth groupoids whose members are ``geometrically reduced'' in a suitable sense\textemdash for example, they do not contain non-triv\-i\-al closed regular kernels of constant dimension (we account for the terminology used here in the appendix)\textemdash and then asking whether the morphisms in \(\RedLGpd\) between groupoids belonging to this collection may be represented by morphisms of a more familiar type, like for instance the generalized homomorphisms of the usual Morita category. This sort of approach runs into a number of technical difficulties. (We shall hint at a couple of them presently.) Even if one is willing to believe that these difficulties may be overcome, the construction underlying Definition~\ref{defn:14A.1.11} has undeniable advantages, such as extreme simplicity, full generality, and a very convincing justification at both intuitive and technical level.

To begin with, we point out that all of the Lie groupoids occurring in the preceding examples \ref{exm:1}~and \ref{exm:2} are ``geometrically reduced'' in the above sense. [Actually, this is true of Example~\ref{exm:2}(b) only when the set \(\{\omega\neq0\}\) is dense within \(\R\). Ending up with only ``geometrically reduced'' groupoids was the main reason behind the apparently unnecessary complications in that example.] On the basis of those examples, it seems impossible to dismiss the notion of natural congruence altogether. Furthermore, even if we agree to talk only about ``geometrically reduced'' groupoids, it is still necessary to show that for each Lie groupoid a natural ``geometrically reduced model'' may be produced in a systematic way. It is not at all obvious how to do that in general. A sample of the pathologies which arise in practice and which make the task arduous is given in our next example.

\begin{exm}\label{exm:4} Let \(L\to M\) be an arbitrary complex line bundle of class \(C^\infty\) over a connected, smooth manifold \(M\). The projection down to \(M\) of a bundle element \(l\in L\) will be denoted by \([l]\). Let \(\omega:M\to\R\) be an arbitrary real-val\-ued function of class \(C^\infty\) on \(M\). Generalizing the construction given in Example~\ref{exm:2}(b), we let the additive group of the real numbers \(\R=(\R,+)\) act on the manifold \(L\) by fiberwise rotations of frequency \(\omega\)\textemdash in other words, we set \(\theta\cdot l=e^{i\omega([l])\theta}l\) for all \(\theta\in\R\emphpunct,l\in L\)\textemdash and then form the translation groupoid \(\R\ltimes L\rightrightarrows L\). The totally isotropic subgroupoid
\begin{equation*}
 K=\setofall*(\theta,l)\suchthat\omega([l])\neq0\et\theta\in2\pi\omega([l])^{-1}\Z\endsetofall\cup\setofall(0,l)\suchthat l\in L\endsetofall
\end{equation*}
is evidently normal, closed, and smooth since it can be parameterized by means of local \(C^\infty\) sections to the groupoid source projection \((\theta,l)\mapsto l\). By the theory of regular kernels (reviewed at the beginning of Appendix~\ref{sec:A}), the quotient groupoid
\begin{equation*}
 \varGamma=(\R\ltimes L)/K\rightrightarrows L
\end{equation*}
is naturally equipped with the structure of a smooth groupoid. We contend that, provided the open set \(U=\{\omega\neq0\}\) and the interior \(V\) of the vanishing locus of \(\omega\) are both non-emp\-ty, the codomain \(\varGamma'\) of any Lie groupoid homomorphism \(\phi:\varGamma\mathrel{%
\xymatrix@1@C=1.3em@M=.em{%
 \ar@{~>}[r]
 &
}}\varGamma'\) belonging to \(\mathcal E\) must contain some non-triv\-i\-al regular kernel which is also closed and of constant dimension. In particular, \(\varGamma\) is unlikely to admit a ``geometrically reduced model'' in the above sense. To begin with, since the existence of such non-triv\-i\-al kernels is a Morita invariant property, by the argument already used in the proof of Proposition~\ref{prop:14A.1.3} it will not be restrictive to assume that \(\phi\) is a full homomorphism covering the identity on \(L\). Let us set \(K'=\ker\phi\). By our hypothesis about \(\phi\), \(K'\) is a regular kernel in \(\varGamma\) and thus \(K'\subset\nef\varGamma\). In particular for every \(l\in L\) such that \(\omega([l])\neq0\) we must have \((K')^l_l\subset\nef\varGamma^l_l=\{1_l\}\). It follows that \((K')^l_l=\{1_l\}\) for every \(l\in L\) not lying over \(V\). Indeed, if \([\theta_0,l_0]\in K'\) is such that \([l_0]\) lies within the closure of the open set \(U\) in \(M\), and so we can find a sequence \(\{l_n\}_{n=1}^\infty\) converging to \(l_0\) in \(L\) such that \(\omega([l_n])\neq0\) for all \(n\), then since by the regularity of \(K'\) there exists some local \(C^\infty\) section \(\sigma\) through \([\theta_0,l_0]\) to the projection \(K'\onto L\) we have
\begin{equation*}
 [\theta_0,l_0]=\sigma(l_0)=\lim_n\sigma(l_n)=\lim_n{}[0,l_n]=[0,l_0]\textpunct.
\end{equation*}
By hypothesis, \(\varGamma'\) is a smooth groupoid, so \(K'\) must be closed and of constant dimension. Since the restriction of \(K'\) over \(\restr L\to U\endrestr\) coincides with the unit bisection of \(\varGamma\), we have \(\dim K'=\dim L\). Thus, any local \(C^\infty\) section to the projection \(K'\onto L\) is also a local parameterization of \(K'\). This immediately implies that the subset
\begin{equation*}
 K''=\tfrac12K'=\setofall*[\tfrac12\theta,l]\suchthat[\theta,l]\in K'\endsetofall\subset\varGamma_{(1)}
\end{equation*}
is itself a closed regular kernel of constant dimension in \(\varGamma\). Since \(K''\supsetneqq K'\), its image \(\phi(K'')\) will be a non-triv\-i\-al closed regular kernel of constant dimension in \(\varGamma'\). \end{exm}

\section{Comparison with the category of reduced orbifolds}\label{sec:4}

The elements of the class \(\mathcal E\) are not the only morphisms of the category \(\LGpd*\) which become invertible under the canonical functor
\begin{equation}
\label{eqn:14A.1.15}
 \LGpd*\longto\catquot{\LGpd*}{\natcong}[\mathcal E^{-1}]\textpunct.
\end{equation}
By way of example, consider any Lie groupoid homomorphism \(\psi:\varGamma\to\varDelta\) for which there exists some element \(\phi:\varDelta\mathrel{%
\xymatrix@1@C=1.3em@M=.em{%
 \ar@{~>}[r]
 &
}}\varDelta'\) of \(\mathcal E\) such that the composition \(\phi\circ\psi\) also lies in \(\mathcal E\). Clearly, \(\psi\) must be a homomorphism in \(\LGpd*\), and its image under \eqref{eqn:14A.1.15} must be invertible. We shall call any element of \(\Mor(\LGpd*)\) with this property an \emph{effective equivalence}. In order to maintain our intuitive interpretation of \(\RedLGpd\) as a category of ``effective transversal geometry types'', we need to make absolutely sure\textemdash among other things\textemdash that the lemmas \ref{lem:14A.1.7}~and \ref{lem:14A.1.8} continue to hold for effective equivalences.

There is a canonical functor \(\LGpd*\subset\LGpd\to\Top\) into the category of topological spaces and continuous maps, which to each Lie groupoid \(\varGamma\rightrightarrows M\) assigns the corresponding orbit space \(M/\varGamma\) and to each Lie groupoid homomorphism \(\psi:\varGamma\to\varDelta\) the (continuous) map of \(M/\varGamma\) into \(N/\varDelta\) induced by \(\psi\). Evidently, any two naturally congruent homomorphisms induce the same map between the orbit spaces. By the universal property of quotient categories \cite[Section~II.8]{MacLane}, we obtain a well-de\-fined functor \(\catquot{\LGpd*}{\natcong}\to\Top\). Next, we make use of the universal property of the localization functor \(\catquot{\LGpd*}{\natcong}\to\catquot{\LGpd*}{\natcong}[\mathcal E^{-1}]\) (compare \cite[Lemma~I.1.2]{GZ}) in combination with Lemma~\ref{lem:14A.1.7} to obtain a functor
\begin{equation*}
\xymatrix@C=2.5em{%
 \LGpd*
 \ar[d]^{\text{\eqref{eqn:14A.1.15}}}
 \ar[r]
 &
	\Top
	\ar@{<--}[dl]!UR
\\
 \RedLGpd
}\end{equation*}
which we agree to call the ``coarse moduli space functor''. We have proved:

\begin{stmt}\label{stmt:14A.1.12} Every effective equivalence of Lie groupoids\/ \(\phi:\varGamma\to\varDelta\) induces a homeomorphism\/ \(M/\varGamma\approxto N/\varDelta\) between the orbit space of\/ \(\varGamma\) and the orbit space of\/ \(\varDelta\). \end{stmt}

The generalization to effective equivalences of Lemma~\ref{lem:14A.1.8} and of its consequences, although conceptually not more involved than the generalization of Lemma~\ref{lem:14A.1.7} which we have just obtained, requires some extra preparation work. For every pair of Lie groupoids \(\varGamma\rightrightarrows M\emphpunct,\varDelta\rightrightarrows N\), let us define \(\mathfrak S_{\varGamma,\varDelta}\) to be the set of all 4-tuples \((x,y;\theta,\lambda)\) consisting of a base point \(x\in M\), a base point \(y\in N\), a Lie group homomorphism \(\theta:\act\varGamma^x_x\to\act\varDelta^y_y\) and a \(\theta\)\mdash equivariant linear map \(\lambda:\ttangent x\varGamma\to\ttangent y\varDelta\). We have two obvious maps \(\alpha_{\varGamma,\varDelta}\) and \(\beta_{\varGamma,\varDelta}\) of \(\mathfrak S_{\varGamma,\varDelta}\) into respectively \(M\) and \(N\) and, for each triplet of Lie groupoids \(\varGamma,\varGamma',\varGamma''\), an obvious composition operation%
\begin{subequations}
\label{eqn:14A.1.17}
\begin{equation}
\label{eqn:14A.1.17a}
 \mathfrak S_{\varGamma',\varGamma''}\ftimes{\alpha_{\varGamma',\varGamma''}}{\beta_{\varGamma,\varGamma'}}\mathfrak S_{\varGamma,\varGamma'}\longto\mathfrak S_{\varGamma,\varGamma''}\mathpunct.
\end{equation}
Let \(S_{\varGamma,\varDelta}\) denote the quotient of the set \(\mathfrak S_{\varGamma,\varDelta}\) with respect to the equivalence relation
\begin{multline}
\label{eqn:14A.1.17b}
 (x,y;\theta,\lambda)\sim(x',y';\theta',\lambda')
 \aeq
 \exists g\in\varGamma(x,x')\emphpunct{}\exists h\in\varDelta(y,y')\emphpunct{}[\theta'\circ c_g=c_h\circ\theta%
\\%
 \et \lambda'\circ\varepsilon(g)=\varepsilon(h)\circ\lambda]\textpunct,
\end{multline}
where \(c_g:\act\varGamma^x_x\simto\act\varGamma^{x'}_{x'}\) and similarly \(c_h\) mean ``conjugation'', and where \(\varepsilon(g)\) and \(\varepsilon(h)\) as usual mean ``effect''. Evidently, we have induced maps*
\begin{equation}
\label{eqn:14A.1.17c}
 M/\varGamma\xfro{a_{\varGamma,\varDelta}}S_{\varGamma,\varDelta}\xto{b_{\varGamma,\varDelta}}N/\varDelta
\quad
 \text{and}
\quad
 S_{\varGamma',\varGamma''}\ftimes{a_{\varGamma',\varGamma''}}{b_{\varGamma,\varGamma'}}S_{\varGamma,\varGamma'}\longto S_{\varGamma,\varGamma''}
\end{equation}
\end{subequations}
[*the quotient composition operation being well defined essentially because of the formulas below, which hold for every isotropic arrow \(g\in\varGamma^x_x\):
\begin{equation*}
 \theta\circ c_g=c_{\theta g}\circ\theta\textpunct;
\quad
 \lambda\circ\varepsilon(g)=\varepsilon(\theta g)\circ\lambda\text{\quad (\em$\theta$\mdash equivariance of $\lambda$\em)].}
\end{equation*}
We introduce an auxiliary category, \(\Skel\), which we call the category of ``transversal skeletons'' of Lie groupoids. Lie groupoids are the objects of \(\Skel\). For every pair of Lie groupoids \(\varGamma,\varDelta\) the hom-set \(\Skel(\varGamma,\varDelta)\) consists of all those global sections \(\sigma:M/\varGamma\to S_{\varGamma,\varDelta}\) to the map \(a_{\varGamma,\varDelta}:S_{\varGamma,\varDelta}\to M/\varGamma\) which have the property that the composition \(b_{\varGamma,\varDelta}\circ\sigma:M/\varGamma\to N/\varDelta\) is continuous. The composition of morphisms in \(\Skel\) is defined in terms of the operation~\eqref{eqn:14A.1.17c} in the evident way. We have a canonical functor \(\LGpd*\to\Skel\) which to each Lie groupoid homomorphism \(\psi:\varGamma\to\varDelta\) assigns the global \(a_{\varGamma,\varDelta}\)\mdash section \([x]\mapsto[x,\psi x;\act\psi^x_x,\ttangent x\psi]\). Reasoning as we did before,%
\footnote{\label{ftn:place}%
This is the place where Lemma~\ref{lem:natural} (or rather its obvious generalization to natural congruences) is needed.%
} %
we conclude that there must be a factorization of this functor through \(\RedLGpd\).
\begin{equation*}
\xymatrix@C=2.5em{%
 \LGpd*
 \ar[d]
 \ar[r]
 &
	\Skel
	\ar@{<--}[dl]!UR
\\
 \RedLGpd
}\end{equation*}

\begin{stmt}\label{stmt:14A.1.13} For any effective equivalence of Lie groupoids\/ \(\phi:\varGamma\to\varDelta\) and for any base point\/ \(x\) of\/ \(\varGamma\), the Lie group homomorphism\/ \(\act\phi^x_x:\act\varGamma^x_x\to\act\varDelta^{\phi x}_{\phi x}\) and the\/ \(\act\phi^x_x\)\mdash equivariant linear map\/ \(\ttangent x\phi:\ttangent x\varGamma\to\ttangent{\phi x}\varDelta\) are bijective. Thus, we still have an isomorphism of the form\/ \eqref{eqn:14A.1.9} between the effective infinitesimal model for\/ \(\varGamma\) at\/ \(x\) and that for\/ \(\varDelta\) at\/ \(\phi x\). \end{stmt}

\begin{rmkterm*} In \cite{Tr4}, we called \emph{effective} a Lie groupoid representation whose kernel consists of ineffective isotropic arrows. More in general, we may call \emph{effective} a homomorphism of differentiable groupoids which enjoys the same property. Our terminology `effective equivalence' is consistent with this use of the adjective `effective'. Indeed, by \ref{stmt:14A.1.13}, any effective equivalence is faithfully transversal and hence satisfies the hypotheses of Proposition~\ref{cor:12B.16.4}(b), which then implies the desired property. \end{rmkterm*}

One of the claims we made in the introduction to the present article was that our notion of ``reduced smooth stack'' was going to generalize the notion of reduced orbifold. It is now time to substantiate that claim. In doing this, we shall make essential use of the existence of a calculus of fractions for the localized category \(\catquot{\LGpd*}{\natcong}[\mathcal E^{-1}]\). Without a calculus of fractions, the task of comparing the above-men\-tioned two notions would be very likely an insurmountable mess. We shall adopt the general point of view on orbifolds that is advocated for instance in \cite{MP,Mo1}. In practice, for our purposes this means that the category of reduced orbifolds we want to compare with our category of reduced Lie groupoids is obtained in conformity with the following stepwise procedure. Start with the category of effective orbifold groupoids.%
\footnote{\label{ftn:orbifold}%
For us, an \emph{orbifold groupoid} will be a proper, \'etale, smooth groupoid. Our definition is slightly different from\textemdash but essentially equivalent to\textemdash the definition given in \cite{Mo1}. We remind the reader that a differentiable groupoid is said to be \emph{\'etale}, if its source and its target are \(C^\infty\)\mdash\'etale maps (local diffeomorphisms), and \emph{proper}, if it is Hausdorff and for each compact subset \(K\) of its base manifold the set \(s^{-1}(K)\cap t^{-1}(K)\) is compact.%
} %
Then pass to the quotient category where any two homomorphisms are identified whenever there exists a natural isomorphism connecting them. Finally formally invert the morphisms corresponding to weak equivalences.

We shall place our discussion in a context which is slightly more general than strictly needed. Let \(\effLGpd\subset\LGpd\) denote the full subcategory consisting of all effective Lie groupoids. Since by definition the ineffective subbundle of an effective Lie groupoid is trivial, \(\effLGpd\) is actually a full subcategory of \(\LGpd*\). Moreover, by the same token, the two equivalence relations \(\natiso\) (natural isomorphism) and \(\natcong\) (natural congruence) turn out to coincide when restricted to the morphisms of this subcategory. We thus have a canonical imbedding of categories (i.e., a fully faithful functor which is ``identical'' on objects)
\begin{equation}
\label{eqn:14A.1.19}
 \catquot{\effLGpd}{\natiso}\imbedto\catquot{\LGpd*}{\natcong}\textpunct.
\end{equation}
Let \(\mathcal W\subset\Mor(\LGpd*)\) denote the class of all weak equivalences of Lie groupoids, and set \(\mathcal W_\mathrm{eff}:=\mathcal W\cap\Mor(\effLGpd)\). It is standard routine to check that the localized category \(\catquot{\effLGpd}{\natiso}[\mathcal W_\mathrm{eff}^{-1}]\) admits a calculus of right fractions. By the universal property of localization, the imbedding~\eqref{eqn:14A.1.19} induces a canonical functor
\begin{equation*}
 \catquot\effLGpd\natiso[\mathcal W_\mathrm{eff}^{-1}]\longto\catquot{\LGpd*}{\natcong}[\mathcal E^{-1}]\textpunct.
\end{equation*}

\begin{stmt}\label{stmt:14A.1.14} The above functor is fully faithful and hence an imbedding of categories. \end{stmt}

\begin{proof} ({\it Fullness.}) Let a span be given \(\varDelta\mathrel{%
\xymatrix@1@C=2.5em@M=.em{%
 \ar@{<-}[r]^{\natcongclass\psi}
 &
}}\varGamma'\mathrel{%
\xymatrix@1@C=2.5em@M=.em{%
 \ar@{~>}[r]^{\natcongclass\phi\in\mathcal E}
 &
}}\varGamma\) representing a morphism in \(\RedLGpd\) between two given effective Lie groupoids \(\varGamma\) and \(\varDelta\). Since \(\phi\in\mathcal E\) is transversal, we must have \(\phi(\nef\varGamma')\subset\nef\varGamma\sidetext(=1\text{ because $\varGamma$ is effective})\). Hence \(\nef\varGamma'\subset\ker\phi\). Since in addition \(\phi\) is {\em faithfully}\/ transversal, for every isotropic arrow \(g'\) in \(\varGamma'\) we must have \(\phi(g')\in\nef\varGamma\seq g'\in\nef\varGamma'\), whence a~fortiori \(\ker\phi\subset\nef\varGamma'\). Thus \(\ker\phi=\nef\varGamma'\). Moreover, since \(\psi\in\Mor(\LGpd*)\), we must have \(\psi(\nef\varGamma')\subset\nef\varDelta\sidetext(=1\text{ because $\varDelta$ is effective})\) and therefore \(\ker\psi\supset\nef\varGamma'=\ker\phi\). Now, if as in the proof of Proposition~\ref{prop:14A.1.3} we decompose \(\phi\) into an epimorphism identical on the bases \(\phi'\) followed by a weak equivalence \(\pi\), then from the first homomorphism theorem for Lie groupoids it follows that \(\psi\) admits a unique factorization \(\psi=\psi'\circ\phi'\) through \(\phi'\). It is evident that the span \((\psi',\pi)\) [in which \(\pi\in\mathcal W_\mathrm{eff}\) because \(\varGamma\) is effective and effectiveness is a Morita invariant property] represents the same morphism in \(\RedLGpd\) as \((\psi,\phi)\) does.
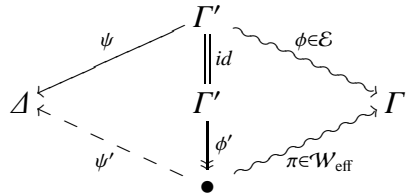
\begin{equation*}
\xymatrix@C=4.5em@R=3.4ex{%
 &
	\varGamma'
	\ar[dl]_\psi
	\ar@{~>}[dr]^{\phi\in\mathcal E}
\\
 \varDelta
 &
	\varGamma'
	\ar@{=}[u]_\id
	\ar@{->>}[d]^{\phi'}
	&
		\varGamma
\\
 &
	*+<1.6ex>{\,\bullet\,}
	\ar@{-->}[ul]^{\psi'}
	\ar@{~>}[ur]_{\pi\in\mathcal W_\mathrm{eff}}
}\end{equation*}

({\it Faithfulness.}) Suppose given a commutative diagram in \(\catquot{\LGpd*}{\natcong}\) of the form
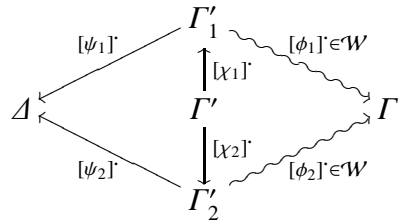
\begin{equation*}
\xymatrix@C=4.4em@R=3.4ex{%
 &
	\varGamma'_1
	\ar[dl]_{\natcongclass{\psi_1}}
	\ar@{~>}[dr]^{\mkern 9mu\natcongclass{\phi_1}\in\mathcal W}
\\
 \varDelta
 &
	\varGamma'
	\ar[u]_(.4){\natcongclass{\chi_1}}
	\ar[d]^(.45){\natcongclass{\chi_2}}
	&
		\varGamma
\\
 &
	\varGamma'_2
	\ar[ul]^{\natcongclass{\psi_2}}
	\ar@{~>}[ur]_{\mkern 9mu\natcongclass{\phi_2}\in\mathcal W}
}\end{equation*}
where\emphpunct: (a)~\(\varGamma,\varGamma'_1,\varGamma'_2,\varDelta\in\Ob(\effLGpd)\)\emphpunct; (b)~\(\phi_1\) and \(\phi_2\) are weak equivalences\emphpunct; (c)~\(\natcongclass{\phi_1\circ\chi_1}=\natcongclass{\phi_2\circ\chi_2}=\natcongclass\phi\) where \(\phi\in\mathcal E\). To begin with, notice that since \(\nef\varGamma=1\) the condition \(\phi_i\circ\chi_i\natcong\phi\) (natural congruence) is equivalent to the condition \(\phi_i\circ\chi_i\natiso\phi\) (natural isomorphism). Similarly, since \(\nef\varDelta=1\), we must have \(\psi_1\circ\chi_1\natiso\psi_2\circ\chi_2\). Now each \(\phi_i\) is a faithful functor, hence for every isotropic arrow \(g'\) in \(\varGamma'\) we must have \(\chi_i(g')=1\aeq\phi_i\chi_i(g')=1\aeq\phi(g')=1\). Thus \(\ker\chi_i=\ker\phi\sidetext(=\nef\varGamma'\text{, as noticed in the previous paragraph})\). Write \(\phi\) as the composition of an epimorphism identical on the bases \(\phi':\varGamma'\onto\varGamma''\) with a weak equivalence \(\pi:\varGamma''\to\varGamma\). Necessarily \(\varGamma''\) is effective, because so is \(\varGamma\). By the first homomorphism theorem for Lie groupoids, we have a unique factorization \(\chi_i=\chi_i'\circ\phi'\) of each \(\chi_i\) through \(\phi'\). Trivially, each \(\chi_i'\) must be a homomorphism in \(\LGpd*\). Now in general
\begin{equation*}
 \alpha\circ\phi'\natiso\beta\circ\phi'
\quad
 \text{implies}
\quad
 \alpha\natiso\beta
\end{equation*}
for any pair of homomorphisms \(\alpha,\beta\) from \(\varGamma''\) into any given other Lie groupoid. [Since \(\phi'\) is identical on the bases, any natural isomorphism \(\tau:\alpha\circ\phi'\natisoto\beta\circ\phi'\) may also be interpreted as a natural isomorphism \(\tau:\alpha\natisoto\beta\).] Since \((\phi_i\circ\chi_i')\circ\phi'=\phi_i\circ\chi_i\natiso\phi=\pi\circ\phi'\), by this observation we see that \(\phi_i\circ\chi_i'\natiso\pi\in\mathcal W\). By the same token, \(\psi_1\circ\chi_1'\natiso\psi_2\circ\chi_2'\). We thus obtain the following commutative diagram in \(\catquot{\effLGpd}{\natiso}\).
\begin{equation*}
 \begin{split}
\xymatrix@C=4.3em@R=3.4ex{%
 &
	\varGamma'_1
	\ar[dl]_{\natisoclass{\psi_1}}
	\ar@{~>}[dr]^{\mkern 9mu\natisoclass{\phi_1}}
\\
 \varDelta
 &
	\varGamma''
	\ar[u]_(.4){\natisoclass{\chi_1'}}
	\ar[d]^(.45){\natisoclass{\chi_2'}}
	\ar@{~>}[r]^(.45){\natisoclass{\pi}}
	&
		\varGamma
\\
 &
	\varGamma'_2
	\ar[ul]^{\natisoclass{\psi_2}}
	\ar@{~>}[ur]_{\mkern 9mu\natisoclass{\phi_2}}
}\end{split}\qquad \text(\phi_1,\phi_2,\pi\in\mathcal W_\mathrm{eff}\text)
\end{equation*}
(The notation \([\void]\) is supposed to indicate the \(\natiso\)\mdash class of a homomorphism.) \end{proof}

Let \(\efforbGpd\subset\effLGpd\) denote the full subcategory with objects all effective orbifold groupoids. (Recall that by an orbifold groupoid we mean a proper \'etale Lie groupoid; cf.~Footnote~\ref{ftn:orbifold}.) By the universal property of localization, we have canonical functors
\begin{equation*}
 \catquot\efforbGpd\natiso[\mathcal W_\mathrm{efforb}^{-1}]\longto\catquot\effLGpd\natiso[\mathcal W_\mathrm{eff}^{-1}]\longto\catquot\LGpd\natiso[\mathcal W^{-1}]\textpunct;
\end{equation*}
here of course we have set \(\mathcal W_\mathrm{efforb}:=\mathcal W\cap\Mor(\efforbGpd)\). Each one of the two localized categories at the extremes of the sequence admits a calculus of right fractions. The right-hand canonical functor is trivially an imbedding, essentially because effectiveness is a Morita invariant property. The other canonical functor is also an imbedding. This is almost as trivial to see, by using the fact that any foliation groupoid is the codomain of a weak equivalence with domain an \'etale groupoid.%
\footnote{\label{ftn:foliation}%
By a \emph{foliation groupoid}, in general, we mean a differentiable groupoid of constant dimension which has only discrete (i.e., zero-di\-men\-sion\-al) isotropy groups. For any foliation groupoid \(\varGamma\rightrightarrows X\) there is some integer \(0\leqq r\leqq\dim X\) such that \(\dim\ttangent x\varGamma=r\) for all \(x\in X\). Equivalently, the \(\varGamma\)\mdash orbits all have the same dimension \(\dim O^\varGamma_x=\dim X-r\) throughout \(X\). In fact
\begin{equation*}
 r=2\dim X-\dim\varGamma\textpunct.
\end{equation*}
When \(\varGamma\rightrightarrows X\) is a {\em smooth} foliation groupoid, we can find a complete transversal \(\incl_T:T\to X\) with domain a smooth manifold \(T\) of dimension \(r\). The pullback groupoid \(\varPi=\incl_T^*\varGamma\rightrightarrows T\) will be a smooth groupoid with all orbits zero-di\-men\-sion\-al and all isotropy groups discrete. Clearly, any such groupoid must be \'etale.%
} %
Combining \ref{stmt:14A.1.14} with the above remarks, we obtain:

\begin{cor}\label{cor:comparison} The canonical functor of localized categories
\begin{equation*}
 \RedOrb:=\catquot\efforbGpd\natiso[\mathcal W_\mathrm{efforb}^{-1}]\longto\catquot{\LGpd*}{\natcong}[\mathcal E^{-1}]=:\RedLGpd
\end{equation*}
imbeds the category of reduced orbifolds into that of reduced Lie groupoids. \end{cor}

\appendix

\section{Remarks on the property of second countability}\label{sec:A}

The present appendix consists of substantially two parts. In the first part, we review a number of standard facts concerning congruences, quotients and kernels in the context of smooth groupoids. The whole part is essentially a straightforward exercise relying on Go\-de\-ment's theorem \cite{Se} and on the basic structure theory of differentiable groupoids (nothing beyond the material recollected at the beginning of Section~\ref{sec:1}). The reader may consult \cite{MaH} for a comprehensive discussion on the topic. The second part assembles a few results which generalize a well-known basic fact in the elementary theory of Lie groups\textemdash namely, that any bijective Lie-group homomorphism is a diffeomorphism\textemdash to Lie groupoids along various directions. Surprisingly, we could find no hint at these results in the literature. They appear to have been overlooked. A possible explanation is that they all depend in an essential way on the property of second countability, which is part of the specific notion of Lie groupoid we adopt here and also part of the standard notion of Lie group but is usually glossed over in most of the literature on Lie groupoids.

We shall say that a homomorphism \(\phi:\varGamma\to\varDelta\) between two arbitrary differentiable groupoids is an \emph{epimorphism}, in symbols `\(\phi:\varGamma\onto\varDelta\)', if the map \(\phi_{(1)}:\varGamma_{(1)}\to\varDelta_{(1)}\) induced by \(\phi\) between the manifolds of arrows is a surjective submersion. Necessarily then the map \(\phi_{(0)}:\varGamma_{(0)}\to\varDelta_{(0)}\) induced by \(\phi\) between the bases is also a surjective submersion. We shall say that \(\phi\) \emph{covers the identity} over a differentiable manifold \(X\) if \(\varGamma_{(0)}=\varDelta_{(0)}=X\) and \(\phi_{(0)}=\id_X\). We shall call \(\phi\) a \emph{monomorphism} if \(\phi_{(1)}\) is an injective immersion, and we shall express this circumstance symbolically by writing `\(\phi:\varGamma\into\varDelta\)'. Necessarily then \(\phi_{(0)}\) is also an injective immersion.%
\footnote{\label{ftn:conflict}%
Unfortunately there is some conflict between the terminology we are introducing here and the usual categorical nonsense. Our outlook is geometrical, in this case, rather than categorical. We have tried to keep our terminology as consistent as possible with that commonly used in Lie theory.}

Let now \(\varGamma\) be an arbitrary (small) category. Recall that a \emph{congruence} \(R\) on \(\varGamma\) is an equivalence relation \(R\subset\varGamma_{(1)}\times\varGamma_{(1)}\) on the arrows of \(\varGamma\) which enjoys the two properties listed below, where we write \(g_1\equiv g_2\) (\(R\) being understood) instead of \((g_1,g_2)\in R\).
\begin{enumerate} \def\labelenumi{(\roman{enumi})} \itemsep=0pt
 \item\(g_1\equiv g_2\seq(sg_1=sg_2\et tg_1=tg_2)\).
 \item\(g_1'\equiv g_2'\seq g''g_1'g\equiv g''g_2'g\).
\end{enumerate}
Given a congruence \(R\) on \(\varGamma\) there is on the quotient set \(\varGamma_{(1)}/R\) a unique structure of category with the same objects as \(\varGamma\) such that the quotient projection \(\pr_{(1)}:\varGamma_{(1)}\onto\varGamma_{(1)}/R\) becomes a functor covering the identity. (Compare \cite[II.8]{MacLane}.) The resulting \emph{quotient category} shall be denoted by \(\varGamma/R\) hereafter. Clearly \(\varGamma/R\) will be a groupoid whenever \(\varGamma\) is. Now suppose \(\varGamma\) is a differentiable groupoid. We shall say that a congruence \(R\) on \(\varGamma\) is \emph{regular} if \(R\) is a regular equivalence relation on the differentiable manifold \(\varGamma_{(1)}\) [viz.~\(R\subset\varGamma_{(1)}\times\varGamma_{(1)}\) is a differentiable submanifold and the projection onto the second factor restricts to a submersion of \(R\) onto \(\varGamma_{(1)}\)]. For any such congruence there exists on the quotient groupoid \(\varGamma/R\) a unique differentiable groupoid structure with the property that the projection functor \(\pr:\varGamma\onto\varGamma/R\) becomes an epimorphism which covers the identity. The expected universal property holds. The quotient differentiable groupoid \(\varGamma/R\) will be Hausdorff if, and only if, \(R\subset\varGamma_{(1)}\times\varGamma_{(1)}\) is a closed submanifold. If \(\varGamma\) is a second countable groupoid then the same will be true of \(\varGamma/R\).

We define the \emph{kernel} of an arbitrary homomorphism of differentiable groupoids \(\phi:\varGamma\to\varDelta\) to be the (abstract, set-the\-o\-ret\-ic) subgroupoid of \(\varGamma\)
\begin{equation*}
 \ker\phi:=\setofall*g\in\varGamma_{(1)}\suchthat sg=tg\et\phi(g)\in u(\varDelta_{(0)})\endsetofall\textpunct.
\end{equation*}
In a broader sense, by a \emph{kernel} in an arbitrary groupoid \(\varGamma\) we shall mean a totally isotropic subgroupoid which contains all the units of \(\varGamma\) and is normal (cf.~Footnote~\ref{ftn:normal} on page~\pageref{ftn:normal}). Of course, the kernel of a homomorphism turns out to be a kernel in this sense. It is not hard to show that if \(\phi:\varGamma\onto\varDelta\) is an epimorphism of differentiable groupoids which covers the identity then \(\ker\phi\) is actually a \emph{regular} kernel in \(\varGamma\), that is to say, a kernel which is also a differentiable subgroupoid of \(\varGamma\).%
\footnote{\label{ftn:subgroupoid}%
A \emph{differentiable subgroupoid} of \(\varGamma\) is a subgroupoid \(\varGamma'\) (in the abstract, set-the\-o\-ret\-ic sense) such that \({\varGamma'}_{(1)}\subset\varGamma_{(1)}\) is a differentiable submanifold and such that the source map of \(\varGamma\) restricts to a submersion of \({\varGamma'}_{(1)}\) onto a differentiable submanifold of \(\varGamma_{(0)}\). With the induced differentiable structure, a differentiable subgroupoid becomes a differentiable groupoid in its own right.%
} %
Moreover, under the same assumption, if the groupoid \(\varDelta\) is Hausdorff then \(\ker\phi\) is a closed subset of the manifold \(\varGamma_{(1)}\). Conversely, let \(K\) be an arbitrary kernel in a given differentiable groupoid \(\varGamma\rightrightarrows X\). The equivalence relation \(R_K\) on the arrows of \(\varGamma\) defined by
\begin{equation*}
 g_1R_Kg_2\aeq(tg_1=tg_2\et g_2^{-1}g_1\in K)
\end{equation*}
is an (abstract, categorical) congruence on \(\varGamma\). In this context, we shall write \(\varGamma/K\) for the quotient groupoid \(\varGamma/R_K\). The congruence \(R_K\) is regular if, and only if, the kernel \(K\) is regular. Whenever \(K\) is regular and the base \(X\) is Hausdorff, the subset \(R_K\subset\varGamma_{(1)}\times\varGamma_{(1)}\) is closed if, and only if, so is the subset \(K\subset\varGamma_{(1)}\). Notice that if the groupoid \(\varGamma\rightrightarrows X\) is of constant dimension and the kernel \(K\) is regular then, as a groupoid over \(X\), \(K\) is of constant dimension if, and only if, the same is true of the quotient \(\varGamma/K\), in which case \[\dim\varGamma/K=\dim\varGamma-\dim K+\dim X\textpunct.\] It follows that for any regular kernel \(K\) in a smooth groupoid \(\varGamma\) the quotient differentiable groupoid \(\varGamma/K\) is smooth if, and only if, \(K\) is closed and of constant dimension.

It is a well-known fact in the elementary theory of Lie groups that any bijective Lie-group homomorphism must be a diffeomorphism and hence an isomorphism; compare \cite[Exercise~2.22(2)]{BtD}. The relevant property of Lie groups, here, is second countability. In fact, the statement in question is false for general (i.e., non-Lie) differentiable groups. By way of example, let \(G\) be the one-di\-men\-sion\-al differentiable group obtained by endowing the additive group of euclidean 2-space \((\R^2,+)\) with the one-di\-men\-sion\-al differentiable structure resulting from the identification
\begin{equation}
\label{eqn:G}
\textstyle%
 \R^2=\coprod\limits_{t\in\R}\R\times\{t\}\textpunct,
\end{equation}
where the right-hand side denotes the disjoint union of uncountably many copies of \(\R\). The identifying map itself provides a bijective homomorphism of differentiable groups between \(G\) and \((\R^2,+)\sidetext[=\text{ standard, two-di\-men\-sion\-al, euclidean Lie group}]\) which is certainly not an isomorphism (not even a homeomorphism).

\begin{lem}\label{lem:12B.5.1} Suppose\/ \(f:X\to Y\) is a map of class\/ \(C^\infty\) from a second countable differentiable manifold of constant dimension\/ \(X\) into an arbitrary differentiable manifold\/ \(Y\). The following implications are true:
\begin{enumerate} \def\labelenumi{\upshape(\alph{enumi})} \itemsep=0pt
 \item If\/ \(f\) is both immersive and surjective then it is a local diffeomorphism.
 \item If\/ \(f\) is a bijective immersion then it is a (global) diffeomorphism.
\end{enumerate} \end{lem}

\begin{proof} The second implication is an immediate consequence of the first one. In order to prove (a), we have to check that for any point \(x\) in \(X\) the rank \(m=\rk T_xf\sidetext[=\dim X]\) equals the local dimension \(n=\dim_{f(x)}Y\). By considering the restriction of \(f\) to the preimage of a local chart in \(Y\) centered at \(f(x)\), it will be no loss of generality to assume that \(Y=\R^n\) and that \(f(x)\) therein is the origin.

We argue by contradiction. Suppose \(m<n\). Since \(f:X\to\R^n\) is an immersion, we may find an open cover \(\setofall U_i\suchthat i\in I\endsetofall\) of \(X\) so that, for each \(i\in I\), \(f\) restricts to a diffeomorphism of \(U_i\) onto a differentiable submanifold \(f(U_i)\) of \(\R^n\). Since \(\dim f(U_i)=m<n\), each \(f(U_i)\) will be a subset of \(\R^n\) with empty interior. Let us fix a countable basis \(\setofall V_k\suchthat k\in\N\endsetofall\) for the topology of the manifold \(X\). Let \(S\) denote the set of all those \(k\in\N\) for which there is some \(i\in I\) such that \(V_k\Subset U_i\).%
\footnote{\label{ftn:Subset}%
For an arbitrary subset \(A\) of a topological space \(T\) we write \(A\Subset T\) (read: \(A\) is \emph{relatively compact} within \(T\)) as an abbreviation for `\(\clos_TA\) is compact'.%
} %
Clearly \(\setofall V_k\suchthat k\in S\endsetofall\) will be a basis for the topology of \(X\), in particular, \(X=\bigcup_{k\in S}V_k\). By the surjectivity of \(f\),
\begin{equation*}
\textstyle%
 \R^n=f(X)=f\within(\bigcup_{k\in S}V_k)=\bigcup_{k\in S}f(V_k)=\bigcup_{k\in S}\overline{f(V_k)}\textpunct.
\end{equation*}
Now each \(\overline{f(V_k)}=\clos_{\R^n}f(V_k)\) is a closed subset of \(\R^n\) of empty interior, because it is contained in some \(f(U_i)\) [since \(\overline{f(V_k)}=f(\clos_{U_i}V_k)\) when \(V_k\Subset U_i\) by the compactness of \(f(\clos_{U_i}V_k)\)]. But the union of countably many such subsets must be itself of empty interior by \mbox{Baire's} theorem \cite[p.~183]{analysisII}: contradiction. \end{proof}

\begin{prop}\label{prop:12B.5.2} Let\/ \(\phi:\varGamma\to\varDelta\) be a homomorphism of Lie groupoids which covers the identity and which is full as an abstract functor. Then\/ \(\phi\) is an epimorphism. \end{prop}

\begin{proof} Let \(M=\varGamma_{(0)}=\varDelta_{(0)}\) denote the common base of \(\varGamma\) and \(\varDelta\). In the first place, we show that for each point \(x\) in \(M\) the map \(O^\phi_x:O^\varGamma_x\to O^\varDelta_{\phi x=x}\) characterized by either of the equations~\eqref{eqn:12B.3.6} is a (global) diffeomorphism. Since by our assumptions the map \(\phi^x:\varGamma^x\to\varDelta^x\) is surjective, Equation~\eqref{eqn:12B.3.6a} entails that \(O^\phi_x\) is surjective. Similarly, since \(\phi_{(0)}=\id_M\), Equation~\eqref{eqn:12B.3.6b} implies that \(O^\phi_x\) is both injective and immersive. Any orbit of a Lie groupoid is a manifold of constant dimension%
\footnote{\label{ftn:constdim}%
The precise general statement, whose proof we leave as an exercise, is the following\emphpunct: \em Any orbit of a differentiable groupoid over a base of constant dimension is a manifold of constant dimension.\em%
} %
which is also second countable (because so is the corresponding source fiber). The desired conclusion drops out at once from Lemma~\ref{lem:12B.5.1}(b).

Next, we show that each isotropy homomorphism \(\phi^x_x:\varGamma^x_x\to\varDelta^x_x\) is a submersion. Let us set \(G=\varGamma^x_x\), \(H=\varDelta^x_x\) and \(f=\phi^x_x\) for short. By our hypotheses, \(f:G\to H\) is a surjective Lie-group homomorphism. Its kernel \(K\) is a closed normal subgroup of \(G\), the quotient \(G/K\) is a Lie group, and the projection \(\pi:G\onto G/K\) is a submersion. By the first homomorphism theorem for Lie groups, there is a unique homomorphism \(\tilde f:G/K\to H\) such that \(f=\tilde f\circ\pi\). Evidently \(\tilde f\) must be bijective and thus, in view of Lemma~\ref{lem:12B.5.1}(b), a diffeomorphism.

We are now in a position to conclude that each one of the maps \(\phi^x:\varGamma^x\to\varDelta^x\) which \(\phi\) induces between two corresponding source fibers is a submersion. In fact, our claim is a straightforward consequence of what we have already shown, Equation~\eqref{eqn:12B.3.6a}, and the principality of the Lie-group bundles \(\pr^\varGamma_x:\varGamma^x\onto O^\varGamma_x\) and \(\pr^\varDelta_x:\varDelta^x\onto O^\varDelta_x\).

Finally, for every arrow \(g\in\varGamma^x\) the matrix expression
\begin{equation*}
 T_g\phi_{(1)}=%
 \begin{pmatrix}
  T_x\phi_{(0)} & 0
 \\
  *             & T_g\phi^x
 \end{pmatrix}=%
 \begin{pmatrix}
  \id & 0
 \\
  *   & T_g\phi^x
 \end{pmatrix}
\end{equation*}
associated with any choice of splittings to the surjective linear maps \(T_gs^\varGamma\) and \(T_{\phi(g)}s^\varDelta\) makes it evident that the linear map \(T_g\phi_{(1)}\) must be surjective. \end{proof}

Recall that a differentiable groupoid \(\varGamma\rightrightarrows X\) is \emph{locally transitive} if the associated combined source\textendash target map \((s,t):\varGamma\to X\times X\) is submersive, and \emph{transitive} if in addition the same map is surjective. By a \emph{one-orbit} groupoid we shall mean a groupoid whose combined source\textendash target map is surjective.

\begin{prop}\label{prop:one-orbit} Any one-orbit Lie groupoid is transitive. \end{prop}

\begin{proof} A general differentiable groupoid \(\varGamma\rightrightarrows X\) is locally transitive if, and only if, for every base point \(x\in X\) the restricted target map \(t^x:\varGamma^x\to X\) is a submersion. If \(t^x\) is surjective then the orbit immersion \(\incl^\varGamma_x:O^\varGamma_x\into X\) is a bijection. Moreover, if \(\varGamma\) is second countable and \(X\) is of constant dimension then the orbit \(O^\varGamma_x\) is a second countable manifold of constant dimension. We conclude from Lemma~\ref{lem:12B.5.1}(b) that whenever \(\varGamma\rightrightarrows X\) is a one-orbit Lie groupoid the orbit immersion \(\incl^\varGamma_x\) is a diffeomorphism and consequently \(t^x=\incl^\varGamma_x\circ\pr^\varGamma_x\) is a submersion. \end{proof}

\begin{npar}[\em Counterexamples\em]\label{npar:12B.5.3} {\bf(a)}\emphpunct{} Let \(M=N\sqcup N\) denote the disjoint union of two copies of a smooth manifold \(N\). Let \(G\) be any Lie group of positive dimension. Let \(\varGamma\) and \(\varDelta\) respectively denote the trivial Lie-group bundles \((G\times N)\sqcup(G\times N)=G\times M\to M\) and \(G\times N\to N\). The map \((1_G\times\id_N)\sqcup\id_{G\times N}:\varGamma\to\varDelta\sidetext(\text{where $1_G$ indicates the constant endomorphism of }G)\) is a surjective Lie-group\-oid homomorphism for which the conclusion of Proposition~\ref{prop:12B.5.2} fails. This shows that the hypothesis \(\phi_{(0)}=\id\) in that proposition cannot be relaxed in any substantial way.

{\bf(b)}\emphpunct{} Let \(G\) denote the one-di\-men\-sion\-al differentiable group arising from the identification~\eqref{eqn:G}. Consider the two translation groupoids \(\varGamma=G\ltimes\R^2\) and \(\varDelta=(\R^2,+)\ltimes\R^2\), the group action in either case being given by the formula \((s,t)\cdot(a,b)=(s+a,t+b)\). Each of them is a one-orbit differentiable groupoid. However, \(\varGamma\) is not transitive. The isotropy groups of \(\varGamma\) and \(\varDelta\) are all trivial. The ``identical'' homomorphism from \(\varGamma\) onto \(\varDelta\) is bijective and covers the identity over \(\varGamma_{(0)}=\varDelta_{(0)}=\R^2\). However, none of the injective immersions induced between the orbits \(O^\varGamma_x\into O^\varDelta_x\) (albeit bijective) can be a diffeomorphism (for reasons of dimension since \(\dim O^\varGamma_x=1\) whereas \(\dim O^\varDelta_x=2\)). \end{npar}

For the sake of completeness, we record the following useful lemma, which shows that there is essentially no such thing as a theory of differentiable groups beyond the classical theory of Lie groups. The only way in which a general differentiable group may fail to be a Lie group is in having uncountably many connected components, each component being a perfectly nice, smooth manifold.

\begin{lem}\label{lem:12B.5.4} Every connected differentiable group is a Lie group. \end{lem}

\begin{proof} Let \(G\) be any such group. Since translations in \(G\) are diffeomorphisms, \(G\) is of constant dimension. By a standard argument, any differentiable group is Hausdorff. Thus, the claim is essentially all about second countability.

The connectedness of \(G\) entails that if \(U\) is any neighborhood of the unit \(e\) then \(G=\bigcup_{k=1}^\infty U^k\) where \(U^k=U\dotsm U\) (\(k\)\mdash fold product in \(G\)). The product \(ST\) of any dense subsets \(S\subset V\emphpunct,T\subset W\) of two arbitrary subsets \(V\emphpunct,W\) of \(G\) must be dense within the product \(VW\). Indeed, let \(\varv\in V\), \(\varw\in W\) and let \(U\ni\varv\varw\) be any open neighborhood. By the continuity of group multiplication, we may find open neighborhoods \(V'\ni\varv\emphpunct,W'\ni\varw\) so that \(V'W'\subset U\). By density, there will be elements \(s\in S\cap V'\emphpunct,t\in T\cap W'\). Then, \(st\in ST\cap V'W'\subset ST\cap U\).

Let us fix an arbitrary local chart \(\varphi:U\approxto\R^n\) for \(G\) with center at \(e=\varphi^{-1}(0)\). For every \(k\in\N\) let us put \(U_k=\varphi^{-1}\bigl(B_{1/k}(0)\bigr)\), where \(B_{1/k}(0)\) denotes the open ball \(\setofall x\in\R^n\suchthat\norm x<1/k\endsetofall\). Since \(R\approx_\varphi\Q^n\) is a dense subset of \(U\approx_\varphi\R^n\), we conclude from the above that \(R^k\) must be dense within \(U^k\) for all \(k\) and hence that \(\bigcup_{k=1}^\infty R^k\) must be dense within \(\bigcup_{k=1}^\infty U^k=G\). Thus \(G\) has to contain some dense sequence \(\{g_l\}_{l\in\N}\). We contend that \(\setofall V_kg_l\suchthat(k,l)\in\N\times\N\endsetofall\) where \(V_k=U_k\cap U_k^{-1}\) must be a countable basis for the topology of \(G\). Indeed, given any element \(\varw\) of an open subset \(W\) of \(G\) we choose \(k\) so that \(V_kV_k\subset U_kU_k\subset W\varw^{-1}\) and then \(l\) so that \(g_l\in V_k\varw\); then clearly \(\varw\in V_kg_l\subset W\). \end{proof}

\begin{prop}\label{prop:12B.5.5} Let\/ \(\phi:\varGamma\to\varDelta\) be a homomorphism between two arbitrary differentiable groupoids. Suppose that\/ \(\phi\) covers the identity and is faithful as an abstract functor. Then\/ \(\phi\) is a monomorphism. \end{prop}

\begin{proof} In view of the hypothesis \(\phi_{(0)}=\id\), Equation~\eqref{eqn:12B.3.6b} implies that each one of the maps \(O^\phi_x:O^\varGamma_x\to O^\varDelta_x\) induced by \(\phi\) between two corresponding orbits is an injective immersion. As next step, we show that each isotropy homomorphism \(\phi^x_x:\varGamma^x_x\to\varDelta^x_x\) is an immersion. Clearly, it will be enough to show that \(T_{1_x}\phi^x_x\) is an injective linear map. By Lemma~\ref{lem:12B.5.4}, the identity component of \(\varGamma^x_x\) is a Lie group. The claim is then an immediate consequence of the naturality of the exponential map of a Lie group \cite[(3.2), p.~23]{BtD}. The proof proceeds by analogy with that of Proposition~\ref{prop:12B.5.2}. \end{proof}

\begin{cor}\label{cor:12B.5.6} Any fully faithful homomorphism of Lie groupoids which covers the identity is an isomorphism. \qed \end{cor}

{\footnotesize
\bibliographystyle{abbrv}
\bibliography{bib/2014a}

\begin{thebibliography}{10}

\bibitem{ALR}
A.~Adem, J.~Leida, and Y.~Ruan.
\newblock {\em Orbifolds and Stringy Topology}.
\newblock Number 171 in Cambridge Tracts in Mathematics. Cambridge University
  Press, Cambridge, 2007.

\bibitem{BtD}
T.~Br{\"o}cker and T.~tom Dieck.
\newblock {\em Representations of Compact {L}ie Groups}.
\newblock Number~98 in Graduate Texts in Mathematics. Springer-Verlag, New
  York, 1995.
\newblock Corrected reprint of the 1985 translation.

\bibitem{FdH}
M.~L. del Hoyo and R.~L. Fernandes.
\newblock {R}iemannian metrics on {L}ie groupoids.
\newblock Preprint arXiv:\hskip.pt 1404.5989 [math.DG], Apr.~23, 2014.

\bibitem{DFMS}
L.~Dixon, D.~Friedan, E.~Martinec, and S.~Shenker.
\newblock The conformal field theory of orbifolds.
\newblock {\em Nuclear Phys. B}, 282(1):13--73, 1987.

\bibitem{FaG}
B.~Fantechi and L.~G{\"o}ttsche.
\newblock Orbifold cohomology for global quotients.
\newblock {\em Duke Math. J.}, 117(2):197--227, 2003.

\bibitem{FOR}
R.~L. Fernandes, J.-P. Ortega, and T.~S. Ratiu.
\newblock The momentum map in {P}oisson geometry.
\newblock {\em Amer. J. Math.}, 131(5):1261--1310, 2009.

\bibitem{GZ}
P.~Gabriel and M.~Zisman.
\newblock {\em Calculus of Fractions and Homotopy Theory}.
\newblock Number~35 in Er\-geb\-nis\-se der Ma\-the\-ma\-tik und ihrer
  Grenz\-ge\-bie\-te (2.~Fol\-ge). Springer-Verlag, Berlin\textendash
  Heidelberg, 1967.

\bibitem{GiK}
V.~Ginzburg and D.~Kaledin.
\newblock Pois\-son deformations of symplectic quotient singularities.
\newblock {\em Adv. Math.}, 186(1):1--57, 2004.

\bibitem{HeM}
A.~Henriques and D.~S. Metzler.
\newblock Presentations of noneffective orbifolds.
\newblock {\em Trans. Amer. Math. Soc.}, 356(6):2481--2499, 2004.

\bibitem{MaH}
P.~J. Higgins and K.~C.~H. Mackenzie.
\newblock Fibrations and quotients of differentiable groupoids.
\newblock {\em J. London Math. Soc. (2)}, 42(1):101--110, 1990.

\bibitem{JeM}
B.~Jelenc and J.~Mr{\v{c}}un.
\newblock Homotopy sequence of a topological groupoid with a basegroup and an
  obstruction to presentability of proper regular {L}ie groupoids.
\newblock {\em J. Homotopy Relat. Struct.}, 2013.

\bibitem{Ka}
T.~Kawasaki.
\newblock The signature theorem for ${V}$\mdash manifolds.
\newblock {\em Topology}, 17(1):75--83, 1978.

\bibitem{Ko}
{\relax Yu}.~A. Kordyukov.
\newblock Classical and quantum ergodicity on orbifolds.
\newblock {\em Russ. J. Math. Phys.}, 19(3):307--316, 2012.

\bibitem{analysisII}
S.~Lang.
\newblock {\em Analysis~{II}}.
\newblock Addison-Wesley Series in Mathematics. Addison-Wesley, Reading,
  Massachusetts, 1969.

\bibitem{Le}
E.~Lerman.
\newblock Orbifolds as stacks?
\newblock {\em En\-seign. Math. (2)}, 56(3/4):315--363, 2010.

\bibitem{MacLane}
S.~Mac~Lane.
\newblock {\em Categories for the Working Mathematician}.
\newblock Number~5 in Graduate Texts in Mathematics. Springer-Verlag, New York,
  second edition, 1998.

\bibitem{Me}
D.~S. Metzler.
\newblock Topological and smooth stacks.
\newblock Preprint arXiv:\hskip.pt math/0306176 [math.DG], June~10, 2003.

\bibitem{Mo1}
I.~Moerdijk.
\newblock Orbifolds as groupoids: an introduction.
\newblock In {\em Orbifolds in Mathematics and Physics ({M}adison, {WI},
  2001)}, number 310 in Contemp. Math., pages 205--222. Amer. Math. Soc.,
  Providence, RI, 2002.

\bibitem{Mo2}
I.~Moerdijk.
\newblock {L}ie groupoids, gerbes, and non-{A}belian cohomology.
\newblock {\em $K$\mdash Theory}, 28(3):207--258, 2003.

\bibitem{MM}
I.~Moerdijk and J.~Mr{\v{c}}un.
\newblock {\em Introduction to Foliations and {L}ie Groupoids}.
\newblock Number~91 in Cambridge Studies in Advanced Mathematics. Cambridge
  University Press, Cambridge, 2003.

\bibitem{MP}
I.~Moerdijk and D.~A. Pronk.
\newblock Orbifolds, sheaves and groupoids.
\newblock {\em $K$\mdash Theory}, 12(1):3--21, 1997.

\bibitem{Pr}
D.~A. Pronk.
\newblock Etendues and stacks as bicategories of fractions.
\newblock {\em Compositio Math.}, 102(3):243--303, 1996.

\bibitem{Sa}
I.~Satake.
\newblock The {G}auss-{B}onnet theorem for ${V}$\mdash manifolds.
\newblock {\em J. Math. Soc. Japan}, 9:464--492, 1957.

\bibitem{Se}
J.-P. Serre.
\newblock {\em {L}ie Algebras and {L}ie Groups}.
\newblock Number 1500 in Lecture Notes in Mathematics. Springer-Verlag, Berlin,
  2006.
\newblock 1964 Lectures given at Harvard University; corrected fifth printing
  of the second (1992) edition.

\bibitem{Tr2}
G.~Trentinaglia.
\newblock {\em {T}annaka Duality for Proper {L}ie Groupoids}.
\newblock PhD thesis, \mbox{Utrecht} University, Sept. 2008.
\newblock Available for download.

\bibitem{Tr4}
G.~Trentinaglia.
\newblock On the role of effective representations of {L}ie groupoids.
\newblock {\em Adv. Math.}, 225(2):826--858, 2010.

\bibitem{Tr7}
G.~Trentinaglia.
\newblock A fast convergence theorem for nearly multiplicative connections on
  proper {L}ie groupoids.
\newblock Preprint arXiv:\hskip.pt 1403.2071 [math.DG], Mar.~9, 2014.

\end{thebibliography}
}%
\end{document}